\documentclass[12pt,reqno]{amsart}
\usepackage{amssymb,amsmath,amsthm,fullpage,enumerate}
\usepackage{xcolor}
\usepackage{biblatex} 
\usepackage{verbatim}
\usepackage{cases}

\usepackage[OT2,OT1]{fontenc}

\DeclareFontFamily{OT2}{cmr}{\hyphenchar\font45 }
\DeclareFontShape{OT2}{cmr}{m}{n}{<->wncyr10}{}
\DeclareFontShape{OT2}{cmr}{m}{it}{<->wncyi10}{}
\DeclareFontShape{OT2}{cmr}{m}{sc}{<->wncysc10}{}
\DeclareFontShape{OT2}{cmr}{b}{n}{<->wncyb10}{}
\DeclareFontShape{OT2}{cmr}{bx}{n}{<->ssub*wncyr/b/n}{}

\DeclareFontFamily{OT2}{cmss}{\hyphenchar\font45 }
\DeclareFontShape{OT2}{cmss}{m}{n}{<->wncyss10}{}

\DeclareRobustCommand\cyr{\fontencoding{OT2}\selectfont}
\DeclareTextFontCommand{\textcyr}{\cyr}

\DeclareUnicodeCharacter{0306}{}

\newtheorem{theorem}{Theorem}[section]
\newtheorem{prop}[theorem]{Proposition}

\newtheorem{corr}[theorem]{Corollary}

\theoremstyle{definition}

\newtheorem{deff}[theorem]{Definition}

\theoremstyle{remark}
\newtheorem{comm}[theorem]{Remark}

\newcommand{\R}{\mathbb R}

\newcommand{\C}{\mathbb C}

\newcommand{\p}{\partial}

\newcommand{\dbar}{\bar\partial}

\DeclareMathOperator{\im}{Im}
\DeclareMathOperator{\re}{Re}

\newcommand{\bc}{\mathbb{B}}

\newcommand{\sca}{\text{Sc}\,}
\newcommand{\vect}{\text{Vec}\,}

\newcommand{\disk}{\mathbb{D}}

\title{Bicomplex Hardy Classes of Solutions to Beltrami Equations and the Schwarz Boundary Value Problem}

\author{William L. Blair}
\address{Department of Mathematics\\
  The University of Texas at Tyler\\
  Tyler, TX 75799}
\email{wblair@uttyler.edu}



\begin{document}

\begin{abstract}
    We define Hardy classes of bicomplex-valued functions on the complex unit disk which solve bicomplex versions of the Beltrami and related equations. Using representations in terms of their complex-valued counterparts, we show these bicomplex-valued functions recover the boundary behavior associated with the classic holomorphic Hardy spaces. This work generalizes known results for complex-valued functions and continues recent work in the setting of bicomplex analogues of Hardy spaces of both holomorphic and generalized analytic functions. Also, we show Schwarz and Dirichlet boundary value problems associated with the bicomplex Beltrami equation are solvable and provide solution formulas.
\end{abstract}

\maketitle

\section{Introduction}

We work to extend the theory of Hardy spaces and boundary value problems to bicomplex-valued functions on the complex unit disk that solve a bicomplex Beltrami equation.

    In contrast to ordinary differential equations, there is no unifying theory of partial differential equations or their solutions. The study of differential equations that satisfy an ellipticity condition, with its close connection to application and modeling of physical phenomenon, is an ongoing success. In the complex plane, the first-order elliptic equations are reduced to the Beltrami equation (and some other related equations). See \cite{Bo, Vek}.

    The classic Hardy spaces of holomorphic functions on the disk are a triumph of complex function theory. Originally pioneered by the Riesz brothers (after Hardy), Zygmund widely communicated the fundamental results of these function spaces in \cite{Zygmund}. See also \cite{Duren, Koosis, BAF, RudinCR, Rep, Pavlovic} for more modern sources. Our interest in the Hardy spaces is that they are the holomorphic functions with $L^p$ boundary values. Not only do these functions have boundary values in $L^p$, but these functions converge to these boundary values in the corresponding $L^p$ norm. This makes them precisely the class of functions to consider in the pursuit of holomorphic solutions to boundary value problems.

    Recently, Hardy classes of solutions to other first-order partial differential equations were considered as generalizations of the holomorphic Hardy spaces. Two common themes in these classes of functions are:
        \begin{enumerate}
            \item The class of functions is defined to be the collection of solutions to a complex partial differential equation with finite $H^p$ norm, where this is the classic $H^p$ norm of the holomorphic Hardy spaces. 
            \item The class of functions must recover the Hardy space boundary behavior of having $L^p$ boundary values on the circle and convergence to those boundary values in the $L^p$ norm. 
        \end{enumerate}
    Differential equations considered, so far, include the Vekua equation \cite{KlimBook, threedvek, BoundExtVek, PozHardy}, the Beltrami equation \cite{Zins, quasiHardy, quasiHardy2, quasiHardy3, quasiHardy4, quasiHardy5, Klim2014, KlimBook, Klim2016Pathlogical}, the conjugate Beltrami equation \cite{conjbel, BoundedExtremal, CompOp, moreVekHardy}, the general first-order elliptic equation \cite{Klim2016}, the higher-order Cauchy-Riemann equations associated with polyanalytic functions \cite{polyhardy}, the higher-order Vekua equations associated with the meta-analytic functions \cite{metahardy, WB3}, general nonhomogeneous Cauchy-Riemann equations \cite{WB}, and the general higher-order iterated Vekua equations \cite{WBD}.

    The bicomplex numbers are a four real-dimensional (two complex-dimensional) extension, originaly by Segre \cite{Segre}, of the complex numbers that, unlike the quaternions, have a commutative multiplication. However, division is not well defined for the bicomplex numbers, as there are bicomplex non-zero zero divisors. The bicomplex numbers are a useful tool in studying the complex stationary Schr\"odinger equation. Specifically, the differential operator  associated with the complex stationary Schr\"odinger equation is factorized using bicomplex numbers into two first-order operators with one of them associated with a bicomplex version of the Vekua equation. Solutions of this bicomplex Vekua equation defined on the complex unit disk were studied in \cite{CastaKrav, FundBicomplex, KravAPFT, ComplexSchr}. In particular, the author considered Hardy classes of solutions to this bicomplex Vekua equation in \cite{BCHoiv} and showed that they recover the classic Hardy space boundary behavior. See also \cite{BCTransmutation, BCBergman} where Bergman spaces of solutions to the bicomplex Vekua equation were considered. The key feature of solutions to the bicomplex Vekua equation that allows the function theory of the complex Hardy classes of solutions to Vekua equations to be extended is that solutions to the bicomplex Vekua equation can be represented as a linear combination of a solution to a complex Vekua equation and a complex conjugate of a solution to a complex Vekua equation. This motivated the question: ``For what other kinds of bicomplex partial differential equations are solutions representable by solutions of associated complex partial differential equations?" In this work, we show this behavior continues for the bicomplex Beltrami equation and some other related equations. Also, we work to show that two of the classic boundary value problems of the complex plane, the Schwarz and Dirichlet problems, can be solved for the bicomplex Beltrami equation. While this is known for the complex Beltrami equation, see \cite{BeltramiSchwarz}, both of these problems have not been well studied yet for bicomplex-valued functions. See \cite{BCSchwarz} for another consideration of these problems for bicomplex-valued functions.

We describe the layout of the paper. In Section \ref{section: background}, we provide relevant background and results from the literature. In Section \ref{section: beleqn}, we define the bicomplex version of the Beltrami equation, the associated Hardy classes of solutions, and prove the boundary behavior of these functions. We show these classes of functions have representations in terms of functions in Hardy classes of solutions to complex Beltrami equations and recover the boundary behavior of the classic complex holomorphic Hardy spaces. Also, we consider higher-order Beltrami equations in the complex and bicomplex settings. We demonstrate that solutions of these higher-order equations in a Hardy class have representation formulas in terms of solutions to the first-order equation and recover Hardy space boundary behavior. This is the first time Hardy classes of solutions to higher-order Beltrami equations have been considered. In Sections \ref{section: conjbel} and \ref{section: GFOE}, we recover many of the results from Section \ref{section: beleqn} when the bicomplex Beltrami equation is replaced with bicomplex variants of conjugate Beltrami equations or general first-order elliptic equations. In Section \ref{section: bvp}, we consider bicomplex versions of the Schwarz and Dirichlet problem for the bicomplex Beltrami equation, show these problems are solvable, and provide formulas for the solutions.

\section{Background}\label{section: background}

    \subsection{Complex Origins}

        Let $\disk$ denote the complex unit disk centered at the origin, and $\p\disk$ denote its boundary. By $L^p(\disk)$, we mean the Lebesgue space of complex-valued functions with $p$-integrable modulus over $\disk$, and by $C^k(\p \disk)$, the space of $k$-times continuously differentiable functions on $\p\disk$. We denote by $W^{k,p}(\disk)$ the Sobolev spaces of complex-valued functions that along with their derivatives up to order $k$ are in $L^p(\disk)$. We denote by $\mathcal{D}'(\p\disk)$ the distributions on $\p \disk$. To disambiguate from other notions of conjugation that will be presented later, we indicate the complex conjugate of $z = x +iy \in \C$ by $z^* = x-iy$. We use $\frac{\p}{\p z}$ and $\frac{\p}{\p z^*}$ for the usual first-order complex differential operators with respect to the variable $z$ and its complex conjugate, respectively.

        \subsubsection{Holomorphic Functions}

        To begin, we recall the classic definition of complex holomorphic functions defined on the complex unit disk, the associated Hardy spaces, and a theorem that describes the boundary behavior of the Hardy spaces.

        \begin{deff}
            We denote by $Hol(\disk)$ the complex holomorphic functions on the complex unit disk, i.e., the collection of functions $f:\disk\to\C$ such that
            \[
                \frac{\p f}{\p z^*}  = 0.
            \]      
            For $0 < p < \infty$, we denote by $H^p(\disk)$ the complex holomorphic Hardy spaces, i.e., the collections of functions $f \in Hol(\disk)$ such that
            \[
                ||f||_{H^p(\disk)}:= \left( \sup_{0 < r < 1} \frac{1}{2\pi}\int_0^{2\pi} |f(re^{i\theta})|^p\,d\theta \right)^{1/p} < \infty.
            \]      
        \end{deff}

        \begin{theorem}[\cite{Duren, Koosis, Rep, RudinCR, BAF}]\label{bvcon}
        A function $f \in H^p(\disk)$, $0 < p < \infty$, has nontangential boundary values $f_{nt} \in L^p(\partial \disk)$ at almost every point of $\p \disk$, 
        \[
            \lim_{r\nearrow 1} \int_0^{2\pi} |f(re^{i\theta})|^p \, d\theta = \int_0^{2\pi} |f_{nt}(e^{i\theta})|^p \,d\theta,
        \]
        and
        \[
            \lim_{r\nearrow 1} \int_0^{2\pi} |f(re^{i\theta})- f_{nt}(e^{i\theta})|^p \, d\theta = 0.
        \]
        \end{theorem}

        \subsubsection{Beltrami Equations}

        The Beltrami equation is a well studied first-order elliptic partial differential equation in the plane. There is particular interest in the connection between solutions of the Beltrami equation and quasiconformal mappings. See \cite{ellipquasi} for a thorough background on quasiconformal mappings. 
        
        For reference, we define the classic complex Beltrami equation.

        \begin{deff}
            Let $\mu \in L^\infty(\disk)$ where there exists a positive constant $c$ such that $||\mu||_{L^\infty(\disk)} \leq c < 1$. Any equation of the form 
            \[
                \frac{\p w}{\p z^*} = \mu \frac{\p w}{\p z}
            \]
            is called a complex Beltrami equation. 
        \end{deff}

        Next, we define Hardy classes of solutions to Beltrami equations. See \cite{Klim2016, Klim2014, Klim2016Pathlogical} for general consideration of these classes and \cite{quasiHardy, Zins, quasiHardy2, quasiHardy4, quasiHardy5} under the assumption that the functions satisfy the additional requirement of being quasiconformal (or quasiregular). 

        \begin{deff}
            For $0 < p < \infty$ and $\mu \in L^\infty(\disk)$ where there exists a positive constant $c$ such that $||\mu||_{L^\infty(\disk)} \leq c < 1$, we define the complex Beltrami-Hardy spaces $H^p_{Bel,\mu}(\disk)$ to be the collection of functions $w:\disk\to\mathbb{C}$ that solve
            \[
                \frac{\p w}{\p z^*} = \mu \frac{\p w}{\p z}
            \]
            and satisfy  
            \[
            \sup_{0 < r < 1} \int_0^{2\pi} |w(re^{i\theta})|^p\,d\theta  < \infty.
            \]
        \end{deff}

        The next theorem combines many known results from the complex setting that will be referenced throughout the paper. Specifically, the theorem includes the classic inclusion result for the holomorphic Hardy spaces into Bergman spaces of larger exponent follows for the Hardy classes of solutions to Beltrami equations, the theorem extends the boundary behavior of the holomorphic Hardy spaces from Theorem \ref{bvcon} to $H^p_{Bel,\mu}(\disk)$, as well as other classic results about the holomorphic Hardy spaces to $H^p_{Bel,\mu}(\disk)$. The theorem illustrates the significant influence that boundary values have on the function in the interior. 

 \begin{theorem}[Theorems 6.1.2, 6.1.3, 6.1.4, and 6.1.5 \cite{KlimBook}, Theorems 1, 2, 3, and 4 \cite{Klim2016}, Theorem 4.1, Corollary 4.4  \cite{quasiHardy}, \cite{Zins}]\label{thm: belhpinbetterlp}\label{thm: belbvcon}\label{thm: betterbvbetterleb}\label{thm: zeroposmesidentzero}
             For $0 < p < \infty$ and $\mu \in W^{1,s}(\disk)$, $s>2$, where there exists a positive constant $c$ such that $||\mu||_{L^\infty(\disk)} \leq c < 1$, every $w \in H^{p}_{Bel,\mu}(\disk)$ is an element of $L^m(\disk)$, for every $0 < m < 2p$, has a nontangential boundary values $w_{nt} \in L^p(\partial \disk)$ at almost every point of $\p \disk$,
        \[
            \lim_{r\nearrow 1} \int_0^{2\pi} |w(re^{i\theta})|^p \, d\theta = \int_0^{2\pi} |w_{nt}(e^{i\theta})|^p \,d\theta,
        \]
and
        \[
            \lim_{r\nearrow 1} \int_0^{2\pi} |w(re^{i\theta})- w_{nt}(e^{i\theta})|^p \, d\theta = 0.
        \]
Also, a function $w \in H^p_{Bel,\mu}(\disk)$ with nontangential boundary value $w_{nt} \in L^\nu(\p \disk)$, where $\nu > p$, is an element of $H^\nu_{Bel,\mu}(\disk)$, and if $w_{nt}$ vanishes on a set $E \subset \p \disk$ of positive measure, then $w$ is identically equal to zero.
        \end{theorem}

        \subsubsection{Conjugate Beltrami Equations}

        Next, we define the complex conjugate Beltrami equation, the associated Hardy classes of solutions, and follow this by including results from the literature that we generalize in Section \ref{section: conjbel}.

        \begin{deff}
            Let $\mu \in L^\infty(\disk)$ where there exists a positive constant $c$ such that $||\mu||_{L^\infty(\disk)} \leq c < 1$. Any equation of the form 
            \[
                \frac{\p w}{\p z^*} = \mu \frac{\p w^*}{\p  z^*}
            \]
            is called a complex conjugate Beltrami equation. 
        \end{deff}

        \begin{deff}
            For $0 < p < \infty$ and $\mu \in L^\infty(\disk)$ where there exists a positive constant $c$ such that $||\mu||_{L^\infty(\disk)} \leq c < 1$, we define the complex conjugate-Beltrami-Hardy spaces $H^p_{conj,\mu}(\disk)$ to be the collection of functions $w:\disk\to\mathbb{C}$ that solve
            \[
                \frac{\p w}{\p z^*} = \mu \frac{\p w^*}{\p  z^*}
            \]
            and satisfy  
            \[
            \sup_{0 < r < 1} \int_0^{2\pi} |w(re^{i\theta})|^p\,d\theta  < \infty.
            \]
        \end{deff}

         \begin{theorem}[Theorems 1, 2, 3, and 4 \cite{Klim2016}, Proposition 4.3.1 \cite{conjbel}]\label{thm: conjbelinclusioninlebesgue}\label{conjbvcon}\label{thm: conjzeroposmesidentzero}
             For $0 < p < \infty$ and $\mu \in W^{1,s}(\disk)$, $s>2$, where there exists a positive constant $c$ such that $||\mu||_{L^\infty(\disk)} \leq c < 1$, every $w \in H^{p}_{conj,\mu}(\disk)$ is an element of $L^m(\disk)$, for every $0 < m < 2p$, has nontangential boundary values $w_{nt} \in L^p(\partial \disk)$ at almost every point of $\p \disk$, 
        \[
            \lim_{r\nearrow 1} \int_0^{2\pi} |w(re^{i\theta})|^p \, d\theta = \int_0^{2\pi} |w_{nt}(e^{i\theta})|^p \,d\theta,
        \]
        and
        \[
            \lim_{r\nearrow 1} \int_0^{2\pi} |w(re^{i\theta})- w_{nt}(e^{i\theta})|^p \, d\theta = 0.
        \]
Also, a function $w \in H^p_{conj,\mu}(\disk)$ with nontangential boundary value $w_{nt} \in L^\nu(\p \disk)$, where $\nu > p$, is an element of $H^\nu_{conj,\mu}(\disk)$, and if $w_{nt}$ vanishes on a set $E \subset \p \disk$ of positive measure, then $w$ is identically equal to zero.
        \end{theorem}

        \subsubsection{Connection to Vekua Equations}\label{subsubsection: conjbelvekconnect}

        In \cite{conjbel}, a direct connection between solutions of certain complex conjugate Beltrami equations and solutions of a complex Vekua equation are described. See also \cite{BersNir} and \cite{Vek}. We include this connection in the complex setting for reference when we generalize it later.
        
        Let $\nu$ be a real-valued function in $W^{1,\infty}(\disk)$ such that there exists a constant $c$ which satisfies $||\nu||_{L^\infty(\disk)} \leq c < 1$. Define $\sigma: \disk\to\mathbb{R}$ and $\alpha : \disk \to \mathbb{C}$ by 
        \[
                \sigma := \frac{1-\nu}{1+\nu}
        \]
        and
        \[
                \alpha := -\frac{1}{1-\nu^2} \frac{\p \nu}{\p z^*}.
        \]
        Note, by their definition, $\sigma \in W^{1,\infty}(\disk)$ and $\alpha \in L^\infty(\disk)$. Then, by Proposition 3.2.3.1 of \cite{conjbel}, $f: \disk\to\mathbb{C}$ solves the complex conjugate Beltrami equation
        \[
            \frac{\p f}{\p z^*} = \nu \frac{\p f^*}{\p z^*}
        \]
        if and only if 
        the function $w: \disk\to\mathbb{C}$ defined by 
        \[
            w := \frac{f - \nu f^*}{\sqrt{1-\nu^2}}
        \]
        solves the complex Vekua equation
        \begin{equation}\label{eqn: conjbelvekeqn}
            \frac{\p w}{\p z^*} = \alpha w^*.
        \end{equation}
        Also, by Proposition 3.2.3.1 of \cite{conjbel}, $f \in H^p_{conj,\nu}(\disk)$ if and only if $w$ is an element of the complex Vekua-Hardy class $H^p_{0,\alpha}(\disk)$ defined to be the solutions of \eqref{eqn: conjbelvekeqn} that have finite $H^p$ norm. See Definition \ref{deff: bcvekhardy} for the definition of the bicomplex-analogue of the complex Vekua-Hardy classes of functions. Also, see \cite{KlimBook, WBD, PozHardy, CompOp, moreVekHardy} for more information about complex Vekua-Hardy classes.

        \subsubsection{General First Order Elliptic Equations}

        Finally, we define the general first-order elliptic equation, the associated Hardy classes of solutions, and follow this by including results from the literature that we generalize in Section \ref{section: GFOE}.

        \begin{deff}
            Let $A, B \in L^q(\disk)$, $q>2$, and $\mu_1, \mu_2 \in L^\infty(\disk)$ where there exists a positive constant $c$ such that $||\mu_1||_{L^\infty(\disk)} + ||\mu_2||_{L^\infty(\disk)}\leq c < 1$. Any equation of the form 
            \[
                \frac{\p w}{\p z^*} = \mu_1 \frac{\p w}{\p z} + \mu_2 \frac{\p w^*}{\p z^*} + Aw + Bw^*
            \]
            is called a complex general first-order elliptic equation. 
        \end{deff}

        \begin{deff}
            For $0 < p < \infty$ and $A, B \in L^q(\disk)$, $q>2$, and $\mu_1, \mu_2 \in L^\infty(\disk)$ where there exists a positive constant $c$ such that $||\mu_1||_{L^\infty(\disk)} + ||\mu_2||_{L^\infty(\disk)}\leq c < 1$, we define the complex GFOE-Hardy spaces $H^p_{\mu_1,\mu_2,A,B}(\disk)$ to be the collection of functions $w:\disk\to\mathbb{C}$ that solve
            \[
                \frac{\p w}{\p z^*} = \mu_1 \frac{\p w}{\p z} + \mu_2 \frac{\p w^*}{\p z^*} + Aw + B z^*
            \]
            and satisfy  
            \[
            \sup_{0 < r < 1} \int_0^{2\pi} |w(re^{i\theta})|^p\,d\theta  < \infty.
            \]
        \end{deff}

 \begin{theorem}[Theorems 1, 2, 3, and 4 \cite{Klim2016}]\label{thm: gfoehardyinbetterlebesgue}\label{gofebvcon}\label{thm: gfoebvinbetterlebimpliesbetterhp}\label{thm: gfoezerobdimpliesequivzero}
             For $0 < p < \infty$, $A,B \in L^s(\disk)$, and $\mu_1,\mu_2 \in W^{1,s}(\disk)$, $s>2$, where there exists a positive constant $c$ such that $||\mu_1||_{L^\infty(\disk)}  + ||\mu_2||_{L^\infty(\disk)}\leq c < 1$, every $w \in H^{p}_{\mu_1,\mu_2,A,B}(\disk)$ is an element of $L^m(\disk)$, for every $0 < m < 2p$, has a nontangential boundary value $w_{nt} \in L^p(\partial \disk)$ at almost every point of $\p \disk$, 
        \[
            \lim_{r\nearrow 1} \int_0^{2\pi} |w(re^{i\theta})|^p \, d\theta = \int_0^{2\pi} |w_{nt}(e^{i\theta})|^p \,d\theta,
        \]
        and
        \[
            \lim_{r\nearrow 1} \int_0^{2\pi} |w(re^{i\theta})- w_{nt}(e^{i\theta})|^p \, d\theta = 0.
        \]
	Every $w \in H^p_{\mu_1, \mu_2, A, B}(\disk)$ with nontangential boundary value $w_{nt} \in L^\nu(\p \disk)$, where $\nu > p$, is an element of $H^\nu_{ \mu_1, \mu_2, A, B}(\disk)$, and if $w_{nt}$ vanishes on a set $E \subset \p \disk$ of positive measure, then $w$ is identically equal to zero. 
        \end{theorem}

         \subsection{Bicomplex Numbers}

         The bicomplex numbers are a higher-dimensional extension of the complex numbers. See \cite{PriceMulti, BCHolo, ComplexSchr, CastaKrav, BCTransmutation, BCBergman, FundBicomplex} for an extensive background on the bicomplex numbers and where they appear in analysis. 

            \subsubsection{Basics of Bicomplex Numbers}

        We define the bicomplex numbers and recall many results we use from the literature for completeness.

    \begin{deff}
        The bicomplex numbers $\bc$ are the set of elements of $\C^2$ with the usual component-wise addition and multiplication defined by 
        \[
            (u_1, u_2)(v_1, v_2) = (u_1v_1 - u_2v_2, u_1v_2 + u_2 v_2).
        \]
        Elements of $\bc$ can be represented via the identifications
        \[
            (z,0) \in \mathbb{C}^2 \leftrightarrow z \in \mathbb{C}, \quad j = (0,1), \quad (z_1, z_2) \leftrightarrow z_1 + j z_2,
        \]
        with $j^2 = -1$. 
    \end{deff}

    \begin{deff}
        For $z = z_1 + jz_2 \in \bc$, we say $z_1$ is the scalar part of $z$, denoted by $\sca z$, and we say $z_2$ is the vector part of $z$, denoted by $\vect z$. We define the bicomplex conjugate of $z$ to be 
        \[
            \overline{z} = z_1 - j z_2.
        \]
        For $w = x + i y \in \C$, i.e., $x, y \in \R$, we define the bicomplexification of $w$ to be 
        \[
            \widehat{w} = x + jy.
        \]  
    \end{deff}

    The next proposition is a feature of the bicomplex-valued functions that is pivotal to our method for studying the function classes in the sections that follow. We refer to this result as the idempotent representation.

    \begin{prop}[Proposition 1 \cite{BCTransmutation}]\label{everybchasplusandminus}
    Let $w \in \bc$. There exist unique $w^{\pm} \in \C$ such that 
    \[
        w = p^+ w^+ + p^- w^-,
    \]
    where $p^\pm$ are given by
    \[
        p^\pm = \frac{1}{2}(1 \pm ji).
    \]
    Furthermore, 
    \[
    w^{\pm} = \sca w \mp i \vect w.
    \]
    \end{prop}

    \begin{comm}
        The bicomplex numbers $p^\pm$ satisfy the relationships
        \[
            (p^\pm)^2 = p^\pm, \quad p^+ + p^- = 1, \text{ and } \quad p^+p^- = 0.
        \]
    \end{comm}

\begin{deff}
    For $w \in \bc$, we define the bicomplex norm of $w$, denoted by $||\cdot||_{\bc}$,  as
    \[
        ||w||_{\bc} := \sqrt{\frac{|w^+|^2 + |w^-|^2}{2}},
    \]
    where $|w^{\pm}|$ is the complex modulus of $w^\pm$. 
\end{deff}

\begin{comm}
It is immediate from the definition of the bicomplex norm $||\cdot||_{\bc}$ that, for every $w = p^+ w^+ + p^- w^- \in \bc$,   \begin{equation}\label{bcbasicestimates}
        \frac{1}{\sqrt{2}} |w^\pm| \leq ||w||_{\bc} \leq \frac{1}{\sqrt{2}}\left( |w^+| + |w^-|\right), 
    \end{equation}
Also, for $w,v \in \bc$, we have
    \begin{equation}\label{stareqn}
        ||wv||_\bc \leq \sqrt{2} \, ||w||_\bc \, ||v||_\bc.
    \end{equation}
\end{comm}

\begin{deff}
    For a positive real number $p$, we define $L^p(\disk,\bc)$ to be the collection of functions $f: \disk\to\bc$ such that
    \[
        ||f||_{L^p(\disk,\bc)} := \left( \iint_{\disk}||f(z)||^p_\bc\,dx\,dy\right)^{1/p} < \infty.
    \]  
    We define $L^\infty(\disk,\bc)$ to be the collection of functions $f: \disk\to\bc$ such that 
    \[
        ||f||_{L^\infty(\disk,\bc)} := \sup_{z \in \disk} ||f(z)||_{\bc} < \infty.
    \]
    For a nonnegative integer $k$ and $0 < p \leq \infty$, we define $W^{k,p}(\disk,\bc)$ to be the collection of functions $f: \disk\to\bc$ such that $f$ and its derivatives up to order $k$ are in $L^p(\disk,\bc)$. The classes of functions $L^p(\p\disk,\bc)$ and $W^{k,p}(\p\disk,\bc)$ are defined analogously for functions on $\p\disk$. 
\end{deff}

\begin{prop}[Proposition 2.24 \cite{BCAtomic}]\label{prop: pminlpfuncinlp} 
    For $0 < p \leq \infty$, $w = p^+ w^+ + p^- w^- \in L^p(\disk,\bc)$ if and only if $w^\pm \in L^p(\disk)$. The same result holds for $L^p(\p \disk,\bc)$. 
\end{prop}

\begin{deff}\label{bcdbardef}
    We define the bicomplex differential operators $\p$ and $\dbar$ as 
    \[
        \p := \frac{1}{2} \left( \frac{\p}{\p x} - j \frac{\p }{\p y}\right)
    \]
    and
    \[
        \dbar := \frac{1}{2} \left( \frac{\p}{\p x} + j \frac{\p }{\p y}\right).
    \]
\end{deff}

  \begin{comm}\label{remark: idempotentdiffop} Observe, by the definition of $\p$ and $\dbar$, the differential operators also have an idempotent representation
    \[
        \p:= p^+ \frac{\p}{\p z^*} + p^- \frac{\p }{\p z}
    \]
    and 
    \[
        \dbar := p^+ \frac{\p}{\p z} + p^- \frac{\p }{\p z^*}.
    \]
    \end{comm}

    \subsubsection{Bicomplex Holomorphic Functions}

        With the differential operators from Definition \ref{bcdbardef}, we define the bicomplex analogue of holomorphicity for a $\bc$-valued function defined on $\disk$.

        \begin{deff}
            We define the $\bc$-holomorphic functions, denoted by $Hol(\disk,\bc)$, to be the collection of functions $w: \disk \to\bc$ such that
            \begin{equation}\label{eqn: bccreqn}
                \dbar w = 0.
            \end{equation}
        \end{deff}

        \begin{comm}
            Note that this differs from the bicomplex holomorphicity considered in \cite{BCHolo} for a $\bc$-valued function of a $\bc$-variable. In \cite{BCHolo}, functions of a bicomplex variable that are differentiable have idempotent representation where the component functions are both complex holomorphic functions of a single complex (not the same) variable. It is immediate that the $p^-$ component function of a $\bc$-valued function of a single complex variable that satisfies \eqref{eqn: bccreqn} is holomorphic but the $p^+$ component function is antiholomorphic. See also Remark 2.3 of \cite{BCAtomic} (or \cite{FundBicomplex, BCTransmutation, BCBergman}). Solutions of \eqref{eqn: bccreqn} have previously been considered in the context of Hardy space theory in \cite{BCAtomic} and the Schwarz boundary value problem in \cite{BCSchwarz}.  
        \end{comm}

    \subsubsection{Bicomplex Vekua Equations}

        Using the bicomplex $\dbar$ from Definition \ref{bcdbardef}, we also define a Vekua-type equation. 

        \begin{deff}
            Let $A, B \in L^q(\disk,\bc)$, $q>2$. We say that any equation of the form 
            \begin{equation}\label{eqn: bcvekeqn}
                \dbar w = Aw + B\overline{w}
            \end{equation}
            is a $\bc$-Vekua equation. 
        \end{deff}

        The $\bc$-Vekua equations were previously considered as a way to study the complex stationary Schr\"odinger equation in \cite{CastaKrav, BCTransmutation, FundBicomplex, KravAPFT, ComplexSchr}, as well as in the context of Bergman and Hardy space theory in \cite{BCTransmutation, BCBergman, BCHoiv}.

        \begin{comm}
            In the same way as solutions of \eqref{eqn: bccreqn} are a linear combination (with respect to the idempotent elements $p^\pm$) of a holomorphic function and an antiholomorphic function, a solution of \eqref{eqn: bcvekeqn} is a linear combination of a solution of a complex Vekua equation and the complex conjugate of a complex Vekua equation. We include this result for reference below.
        \end{comm}

        \begin{theorem}[Theorem 4.9 \cite{BCHoiv}]\label{bcvekimpliescvek}
            Let $A, B \in L^q(\disk,\bc)$, $q>2$. A function $w:\disk\to\bc$ solves
\[
    \dbar w = Aw + B\overline{w}
\]
if and only if
\[
    \frac{\p (w^+)^*}{\p z^*} = (A^+)^* (w^+)^* + (B^+)^* w^+
\]
and
\[
    \frac{\p w^-}{\p z^*} = A^- w^- + B^- (w^-)^*.
\]
        \end{theorem}

    \subsubsection{Bicomplex Hardy Spaces}

        Next, we state definitions for certain Hardy classes of bicomplex-valued functions on $\disk$. These classes of functions were previously considered in \cite{BCAtomic, BCHoiv, BCHarmVek} and motivate the Hardy classes of bicomplex-valued functions that we examine in Sections \ref{section: beleqn}, \ref{section: conjbel}, and \ref{section: GFOE}.

        \begin{deff}\label{deff: bcholohardy}
            For $0 < p < \infty$, we define the $\bc$-holomorphic Hardy spaces $H^p(\disk,\bc)$ to be the collection of $w \in Hol(\disk,\bc)$ such that 
            \[
                \sup_{0 < r < 1} \int_0^{2\pi} ||w(re^{i\theta})||_{\bc}^p \,d\theta < \infty.
            \]
        \end{deff}

        \begin{deff}\label{deff: bcvekhardy}
            For $0 < p < \infty$ and $A, B \in L^q(\disk,\bc)$, $q>2$, we define the $\bc$-Vekua Hardy classes $H^p_{A,B}(\disk,\bc)$ to be the collection of $w: \disk \to\bc$ that solve
            \[
                \dbar w = Aw + B\overline{w}
            \]
            and satisfy
            \[
                \sup_{0 < r < 1} \int_0^{2\pi} ||w(re^{i\theta})||_{\bc}^p \,d\theta < \infty.
            \]
        \end{deff}

     \subsection{The Schwarz and Dirichlet Boundary Value Problems}

    In Section \ref{section: bvp}, we consider boundary value problems associated with the Beltrami equation that we define in Section \ref{section: beleqn}. Here we record results from the literature that we generalize in Section \ref{section: beleqn} and appeal to in our justifications of the results we present. 

    The classic Schwarz boundary value problem is to find a holomorphic function with prescribed real part on the boundary, i.e., the boundary value problem 
    \[
        \begin{cases}
            \frac{\p w}{\p z^*} = 0, & \text{ in } \disk,\\
            \re\{w\}|_{\p \disk} = g,
        \end{cases}
    \]
    for a prescribed function on the circle $g$. Since imaginary parts of holomorphic functions are only well defined up to addition of constants, the problem is made well defined by requiring the imaginary part has a prescribed pointwise value, i.e., the boundary value problem 
    \[
        \begin{cases}
            \frac{\p w}{\p z^*} = 0, & \text{ in } \disk,\\
            \re\{w\}|_{\p \disk} = g,\\
            \im\{w(0)\} = c,
        \end{cases}
    \]
    for a prescribed function on the circle $g$ and $c \in \R$. See \cite{BegBook,Beg} for background on this problem. 
    
    In \cite{BeltramiSchwarz}, the Schwarz boundary value problem was considered where a Beltrami equation replaces the Cauchy-Riemann equation. Explicit conditions for solvability and formulas for producing solutions are provided by the theorem below.

    \begin{theorem}[Theorem 3.1 \cite{BeltramiSchwarz}]\label{thm: beltramischwarz}
        The Schwarz problem 
        \[
            \begin{cases}
                \frac{\p w}{\p z^*} = \mu \frac{\p w}{\p z} + f, & \text{ in } \disk,\\
                \re\{w\}|_{\p\disk} = \gamma, \\
                \im\{w(0)\} = a,
            \end{cases}
        \]
        for $\mu \in \mathbb{C}$ such that $|\mu| \leq c < 1$, for some constant $c$, $f \in L^p(\disk)$, $p >2$, $\gamma \in C(\p \disk,\R)$, and $a \in \R$, is uniquely solvable and the solution is 
        \begin{align*}
            w(z) 
            &= \varphi(z) 
            + \sum_{k=0}^\infty (-1)^{k+1} \frac{1}{2\pi} \iint_{|\zeta|<1} \left( (-\mu)^k T^k(f+\mu\varphi')(\zeta)\frac{\zeta + z}{\zeta(\zeta -z)} \right.\\
            &\quad\quad\quad\quad +  \left. ((-\mu)^*)^k (T^k(f+\mu\varphi')(\zeta))^*\frac{1-z \zeta^*}{\zeta^*(1-z\zeta^*)} \right)\,d\xi\,d\eta,
        \end{align*}
        where $\zeta = \xi + i \eta$,
        \[
            \varphi(z) = \frac{1}{2\pi i} \int_{|\zeta|=1} \gamma(\zeta) \frac{\zeta + z}{\zeta -z} \frac{d\zeta}{\zeta} + ia,
        \]
        and $T(\cdot)$ is the operator defined by
        \[
            T(g)(z) := -\frac{1}{\pi} \iint_{|\zeta|<1} \frac{g(\zeta)}{(\zeta - z)^2} + \frac{(g(\zeta))^*}{(1-z\zeta^*)^2}\,d\xi\,d\eta.
        \]
    \end{theorem}

    In \cite{WBD}, the author and B. B. Delgado showed existence of a unique solution for the Schwarz boundary value problem on the unit disk when the boundary condition is a distributional boundary value. See the definition and theorem below. 

    \begin{deff}\label{bvcircle} Let $f$ be a function defined on the complex unit disk. We say that $f$ has a boundary value in the sense of distributions $f_b$, also called a distributional boundary value, if, for every $\gamma \in C^\infty(\p \disk)$, the limit
    \[
        \lim_{r \nearrow 1} \int_0^{2\pi} f(re^{i\theta}) \gamma(\theta) \,d\theta 
    \]
    exists. 
\end{deff}

    \begin{theorem}[Theorem 6.9 \cite{WBD}]\label{nonhomogsbvp}
        The Schwarz boundary value problem
        \[\begin{cases}
            \frac{\p w}{\p z^*} = f, &\text{ in } \disk,\\
            \re\{w_b\} = g,\\
            \im\{w(0)\} = c,
        \end{cases}\]
        for $f \in L^1(\disk)$, $g\in \mathcal{D}'(\p \disk)$, and $c\in \mathbb{R}$, is uniquely solved by 
        \[
            w = \frac{1}{2\pi}\langle g, P_r(\theta - \cdot) + iQ_r(\theta - \cdot)\rangle + ic  - \frac{1}{2\pi}\iint_{|\zeta|<1} \left( \frac{f(\zeta)}{\zeta}\,\frac{\zeta + z}{\zeta - z} + \frac{(f(\zeta))^*}{\zeta^*}\,\frac{1+z\zeta^*}{1-z\zeta^*}\right)\,d\xi\,d\eta,
        \]
        where $z = re^{i\theta}$, $\zeta = \xi + i \eta$, and
        \[
            P_r(\theta) = \frac{1-r^2}{1-2r\cos(\theta) + r^2}
        \]
        and 
        \[
            Q_r(\theta) =  \frac{2r\sin(\theta)}{1-2r\cos(\theta) + r^2}
        \]
        are the Poisson kernel and the conjugate Poisson kernel on $\disk$, respectively. 
        \end{theorem}

    In \cite{BCSchwarz}, the author constructed a bicomplex version of the Schwarz boundary value problem. For this problem, a solution must satisfy \eqref{eqn: bccreqn} and the have component functions from its idempotent representation with real parts that satisfy boundary conditions and imaginary parts that satisfy a pointwise evaluation condition. This problem is solvable both in the classical case of continuous boundary functions and in the more general setting of the boundary condition being with respect to boundary values in the sense of distributions.  See the relevant theorem below.

    \begin{theorem}[Theorem 3.3 \cite{BCSchwarz}]\label{thm: generalbcschwar}
    For $b_1, b_2 \in \mathcal{D}'(\p \disk)$, $c_1, c_2 \in \R$, and $f \in L^1(\disk, \bc)$, the bicomplex-Schwarz boundary value problem 
\[
\begin{cases}
    \dbar w = f, & \text{in } \disk,\\
    \re\{w^+_b\} = b_1, \\
    \re\{w^-_b\} = b_2,\\
    \im\{ w^+(0)\}  = c_1, \\
    \im\{ w^-(0) \} = c_2
\end{cases}
\]
is uniquely solved by 
\begin{align*}
w(z) &= p^+ \left( \left[\frac{1}{2\pi} \langle b_1, P_r(\theta-\cdot) + i Q_r(\theta - \cdot)\rangle\right]^* + ic_1  \right)  \\
&\quad\quad + p^- \left(\frac{1}{2\pi} \langle b_2, P_r(\theta-\cdot) + i Q_r(\theta - \cdot)\rangle +ic_2 \right) +T_\bc(f)(z),
\end{align*}
    where the integral operator $T_\bc(\cdot)$, acting on functions $f \in L^1(\disk,\bc)$, is defined by 
\[
    T_\bc(f)(z):= p^+ T_*(f^+)(z) + p^- T(f^-)(z),
\]
with
\[
T(f^-)(z) := -\frac{1}{2 \pi} \iint_D \left(\frac{f^-(\zeta)}{\zeta} \frac{\zeta+z}{\zeta - z} + \frac{(f^-(\zeta))^*}{\zeta^*} \frac{1+z\zeta^*}{1-z\zeta^*} \right)\,d\xi\,d\eta
\]
and
\[
T_*(f^+)(z) := -\frac{1}{2 \pi} \iint_D \left(\frac{f^+(\zeta)}{\zeta^*} \frac{(\zeta+z)^*}{(\zeta - z)^*} + \frac{(f^+(\zeta))^*}{\zeta} \frac{1+z^*\zeta}{1-z^*\zeta} \right) \,d\xi\,d\eta.
\]
    \end{theorem}

In Section \ref{section: bvp}, we consider a Schwarz-type boundary value problem similar to the one in Theorem \ref{thm: generalbcschwar} but where the differential equation is the bicomplex version of the Beltrami equation that we consider in Section \ref{section: beleqn}

The classic Dirichlet boundary value problem seeks a harmonic function with prescribed boundary values. In \cite{BeltramiSchwarz}, a Dirichlet problem is considered that is a Beltrami equation with solutions having a prescribed boundary value function. This problem is shown to be solvable under a special condition. The theorem is included below.

\begin{theorem}[Theorem 4.1 \cite{BeltramiSchwarz}]\label{thm: beltramidirichlet}
    The Dirichlet problem 
    \[
        \begin{cases}
            \frac{\p w}{\p z^*} = \mu \frac{\p w}{\p z} + f, & \text{ in } \disk,\\
            w|_{\p\disk} = \gamma,
        \end{cases}
    \]
    where $\mu \in \mathbb{C}$ such that $|\mu|\leq c < 1$, for some constant $c$, $f \in L^p(\disk)$, $p>2$, and $\gamma \in C(\p\disk,\mathbb{C})$, is solvable if and only if 
    \[
        \frac{1}{2\pi i} \int_{|\zeta| = 1} \gamma(\zeta) \frac{2 + (-\mu)z^* \zeta^*}{1 + (-\mu)z^*\zeta^*} \frac{z^* \,d\zeta}{1-z^*\zeta} = \sum_{k=0}^{\infty} (-1)^k (-\mu)^k \frac{1}{\pi} \iint_{|\zeta|<1} f(\zeta) ((\zeta - z)^*)^k \frac{(z^*)^{k+1} \,d\xi\,d\eta}{(1-z^*\zeta)^{k+1}}
    \]
    and the solution is 
    \[
        w(z) = \frac{1}{2\pi i} \int_{|\zeta| =1} \frac{\gamma(\zeta)}{\zeta - z} \,d\zeta + \frac{1}{2\pi i} \int_{|\zeta| =1} \frac{\gamma(\zeta)}{\zeta - z + c(\zeta - z)^*} \,d\zeta - \frac{1}{\pi} \iint_{|\zeta|<1} \frac{f(\zeta)}{\zeta - z + c(\zeta - z)^*}\,d\xi\,d\eta.
    \]
\end{theorem}

In \cite{BCSchwarz}, the author considered a bicomplex Dirichlet problem for harmonic functions. Since harmonic functions with respect to the bicomplex differential operators $\p$ and $\dbar$ are precisely the same as harmonic functions with respect to the classic complex differential operators $\frac{\p}{\p z}$ and $\frac{\p }{\p z^*}$, the novelty of the Dirichlet problem considered there is with respect to the boundary condition now being with respect to a bicomplex-valued function or distribution. The referenced theorem is recorded below. 

\begin{theorem}[Corollary 4.3 \cite{BCSchwarz}]
    The bicomplex Dirichlet problem
    \[
        \begin{cases}
            \p \dbar f = 0\\
            f_b = g
        \end{cases}
    \]
for $g \in \mathcal{D'}(\p \disk, \bc) := \{ h \in \mathcal{D}'(\p \disk): \langle h, \varphi \rangle \in \bc, \text{ for }\varphi \in C^\infty(\p \disk)\}$, is uniquely solved by 
\[
    f = p^+ \frac{1}{2\pi} \langle g^+, P_r(\theta - \cdot) \rangle + p^- \frac{1}{2\pi} \langle g^-, P_r(\theta - \cdot)  \rangle.
\]
\end{theorem}

In Section \ref{section: bvp}, we prove that a Dirichlet problem with a bicomplex Beltrami equation and bicomplex-valued boundary conditions is solvable and provide a formula for the solution.

\section{The Beltrami Equation}\label{section: beleqn}

In this section, we define a bicomplex version of the classic Beltrami equation using the bicomplex $\dbar$ operator from Definition \ref{bcdbardef}. This mirrors the consideration of the bicomplex Vekua-type equations considered in, for example, \cite{BCTransmutation, BCBergman, FundBicomplex, ComplexSchr,BCHoiv, BCHarmVek} and bicomplex nonhomogeneous Cauchy-Riemann equations in \cite{BCAtomic, BCSchwarz}. Using these equations, we define Hardy classes of bicomplex-valued solutions and show these functions exhibit properties associated with the classical Hardy spaces. This extends the work found in \cite{Zins, quasiHardy, quasiHardy2}, amongst others, in the setting of complex-valued functions.

    \subsection{Definition and Representation}

\begin{deff}
    Let $\mu : \disk \to \bc$ be such that there exists a real constant $c$ and $||\mu||_{L^\infty(\disk,\bc)} \leq c < 1$. We call any equation of the form 
    \[
        \dbar w = \mu \, \p w
    \]
    a $\bc$-Beltrami equation. 
\end{deff}

The next proposition shows that solutions of $\bc$-Beltrami equations are representable, using the idempotent representation, in terms of solutions of $\mathbb{C}$-Beltrami equations. Since the function and the components of the idempotent representation are comparable in the bicomplex norm, this is an invaluable tool for the analysis of solutions to $\bc$-Beltrami equations. This idea was previously used in the $\mu \equiv 0$ case ($\bc$-holomorphic) in \cite{BCAtomic} and in the case of solutions to $\bc$-Vekua equations in \cite{BCHoiv, BCHarmVek}.  

\begin{prop}\label{prop: bcbelimpliescomplexbel}
    Let $\mu : \disk \to \bc$ be such that there exists a real constant $c$ and $||\mu||_{L^\infty(\disk,\bc)} \leq c < 1$. Every solution $w:\disk \to\bc$ of 
    \begin{equation}\label{bcbeltramieqn}
        \dbar w = \mu \, \p w
    \end{equation}
    has the form 
    \[
        w = p^+ w^+ + p^- w^-,
    \]
    where $w^+:\disk\to\mathbb{C}$ solves the $\mathbb{C}$-Beltrami equation
    \[
        \frac{\p (w^+)^*}{\p z^*} = (\mu^+)^* \frac{\p (w^+)^*}{\p z}
    \]
    and $w^- : \disk \to \mathbb{C}$ solves the $\mathbb{C}$-Beltrami equation
    \[
        \frac{\p w^-}{\p z^*} = \mu^- \frac{\p w^-}{\p z}.
    \]
\end{prop}

\begin{proof}
    By Proposition \ref{everybchasplusandminus}, every function $w:\disk\to\bc$ has the form 
    \[
        w = p^+ w^+ + p^- w^-,
    \]  
    where $w^\pm : \disk\to\C$. Similarly, an idempotent representation holds for $\mu$. So, if $w$ solves \eqref{bcbeltramieqn}, then, by  Remark \ref{remark: idempotentdiffop}, we have
    \begin{align*}
        \dbar w &= \mu \p w \\
        \left(p^+\frac{\p}{\p z} + p^- \frac{\p}{\p z^*}\right)(p^+w^+ + p^- w^-) &= (p^+\mu^+ + p^- \mu^-)  \left(p^+\frac{\p}{\p z^*} + p^- \frac{\p}{\p z}\right)(p^+w^+ + p^- w^-)\\
        p^+\frac{\p w^+}{\p z} + p^- \frac{\p w^-}{\p z^*} &= p^+ \mu^+ \frac{\p w^+}{\p z^*} + p^- \mu^- \frac{\p w^-}{\p z}.
    \end{align*}
    Since idempotent representations are unique, it follows that
    \[
        \frac{\p w^+}{\p z} = \mu^+ \frac{\p w^+}{\p z^*}
    \quad \text{and} \quad
        \frac{\p w^-}{\p z^*} = \mu^- \frac{\p w^-}{\p z}.
    \]
    Applying complex conjugation to both sides of the left-hand side equation in the display above, we have
    \[
        \frac{\p (w^+)^*}{\p z^*} = (\mu^+)^* \frac{\p (w^+)^*}{\p z}
    \quad \text{and} \quad
        \frac{\p w^-}{\p z^*} = \mu^- \frac{\p w^-}{\p z}.
    \]
\end{proof}

\subsection{Hardy Spaces}

Next, we define the Hardy classes of solutions to bicomplex Beltrami equations. Note that as in Definitions \ref{deff: bcholohardy} and 
\ref{deff: bcvekhardy}, the size condition is with respect to an $H^p$ norm using the $\bc$-norm instead of complex modulus. 

\begin{deff}
For $\mu : \disk \to \bc$ such that there exists a real constant $c$ and $||\mu||_{L^\infty(\disk,\bc)} \leq c < 1$ and $0 < p < \infty$, we define the $\bc$-Beltrami-Hardy spaces $H^p_{\text{Bel},\mu}(\disk,\bc)$ to be those functions $w: \disk\to\bc$ that solve
    \[
        \dbar w = \mu \, \p w
    \]
and satisfy
    \[
        ||w||_{H^p_{\bc}} := \left( \sup_{0< r < 1} \frac{1}{2\pi}\int_0^{2\pi} ||w(re^{i\theta})||_{\bc}^p\,d\theta\right)^{1/p} < \infty.
    \]
\end{deff}

In \cite{BCAtomic}, the author proved that a function is in a $\bc$-holomorphic Hardy space if and only if the component functions in the idempotent representation are themselves elements of the classical Hardy spaces of complex-valued functions. See Theorem 4.1 of \cite{BCAtomic}. This idea was extended in \cite{BCHoiv} where the author showed that a function is in a Hardy space of solutions to a $\bc$-Vekua equation if and only if the component functions in the idempotent representation are elements of complex-valued Hardy spaces of solutions to the complex Vekua equation, as studied in \cite{KlimBook, PozHardy, CompOp, moreVekHardy, conjbel}. See Theorem 4.10 of \cite{BCHoiv}. The next theorem extends Proposition \ref{prop: bcbelimpliescomplexbel} and shows that the relationship between $\bc$-valued Hardy classes and the associated $\mathbb{C}$-valued Hardy classes is recovered for $H^p_{Bel,\mu}(\disk,\bc)$ functions also. 

\begin{theorem}\label{thm: bcbelhardyrep}
    For $\mu : \disk \to \bc$ such that there exists a real constant $c$ satisfying $||\mu||_{L^\infty(\disk,\bc)} \leq c < 1$ and $0 < p < \infty$, $w = p^+ w^+ + p^- w^- \in H^p_{Bel,\mu}(\disk,\bc)$ if and only if $(w^+)^* \in H^p_{Bel, (\mu^+)^*}(\disk)$ and $w^- \in H^p_{Bel,\mu^-}(\disk)$. 
\end{theorem}

\begin{proof}
    By Proposition \ref{prop: bcbelimpliescomplexbel}, if $w$ solves $\dbar w = \mu \p w$, then 
    \[
        \frac{\p (w^+)^*}{\p z^*} = (\mu^+)^* \frac{\p (w^+)^*}{\p z}
    \quad \text{and} \quad
        \frac{\p w^-}{\p z^*} = \mu^- \frac{\p w^-}{\p z}.
    \]
    By \eqref{bcbasicestimates}, for every $r \in (0,1)$, we have 
    \begin{align*}
    \int_0^{2\pi} |(w^+(re^{i\theta}))^*|^p \,d\theta &= 
        \int_0^{2\pi} |w^+(re^{i\theta})|^p \,d\theta \\
        &\leq 2^{p/2}\int_0^{2\pi} ||w(re^{i\theta})||_\bc^p \,d\theta \\
        &\leq 2^{p/2} \sup_{0 < r< 1} \int_0^{2\pi} ||w(re^{i\theta})||_\bc^p \,d\theta< \infty
    \end{align*}
    and 
    \begin{align*}
         \int_0^{2\pi} |w^-(re^{i\theta})|^p \,d\theta 
         &\leq 2^{p/2}\int_0^{2\pi} ||w(re^{i\theta})||_\bc^p \,d\theta \\
         &\leq 2^{p/2}\sup_{0 < r < 1}\int_0^{2\pi} ||w(re^{i\theta})||_\bc^p \,d\theta < \infty.
    \end{align*}
    Thus, 
    \[
        \sup_{0 < r < 1}  \int_0^{2\pi} |(w^+(re^{i\theta}))^*|^p \,d\theta < \infty,
    \]
    and
    \[
        \sup_{0 < r < 1}\int_0^{2\pi} |w^-(re^{i\theta})|^p \,d\theta < \infty.
    \]
    Therefore, $(w^+)^* \in H^p_{Bel, (\mu^+)^*}(\disk)$ and $w^- \in H^p_{Bel,\mu^-}(\disk)$.

    Now, if $(w^+)^* \in H^p_{Bel, (\mu^+)^*}(\disk)$ and $w^- \in H^p_{Bel,\mu^-}(\disk)$, then $w = p^+ w^+ + p^- w^-$ solves $\dbar w = \mu \p w$ by direct computation. Also, by \eqref{bcbasicestimates}, we have for every $r \in (0,1)$ that 
    \begin{align*}
        \int_0^{2\pi} ||w(re^{i\theta})||_\bc^p \,d\theta 
        &\leq 2^{-p/2} \int_0^{2\pi} (|w^+(re^{i\theta})|+|w^-(re^{i\theta})|)^p\,d\theta \\
        &\leq C_p \left( \int_0^{2\pi} |(w^+(re^{i\theta}))^*|^p \,d\theta + \int_0^{2\pi} |w^-(re^{i\theta})|^p\,d\theta\right) \\
        &\leq C_p \left( \sup_{0 < r < 1}\int_0^{2\pi} |(w^+(re^{i\theta}))^*|^p \,d\theta + \sup_{0 < r < 1}\int_0^{2\pi} |w^-(re^{i\theta})|^p\,d\theta\right) < \infty,
    \end{align*}
    where $C_p$ is a constant that depends on only $p$. Thus, 
    \[
        \sup_{0 < r < 1}\int_0^{2\pi} ||w(re^{i\theta})||_\bc^p \,d\theta < \infty,
    \]
    and $w \in H^p_{Bel,\mu}(\disk,\bc)$. 

\end{proof}

\begin{theorem}\label{thm: bcbelhpinbetterlp}
    For $\mu \in W^{1,s}(\disk,\bc)$, $s>2$, such that there exists a real constant $c$ satisfying $||\mu||_{L^\infty(\disk,\bc)} \leq c < 1$ and $0 < p < \infty$, every $w \in H^p_{Bel,\mu}(\disk,\bc)$ is an element of $L^m(\disk,\bc)$, for every $0 < m < 2p$. 
\end{theorem}

\begin{proof}
By Theorem \ref{thm: bcbelhardyrep}, if $w \in H^p_{Bel,\mu}(\disk,\bc)$, then $(w^+)^* \in H^p_{Bel, (\mu^+)^*}(\disk)$ and $w^- \in H^p_{Bel,\mu^-}(\disk)$. By Theorem \ref{thm: belhpinbetterlp}, since $(w^+)^* \in H^p_{Bel, (\mu^+)^*}(\disk)$ and $w^- \in H^p_{Bel,\mu^-}(\disk)$, it follows that $(w^+)^*,w^- \in L^m(\disk)$, for $0 < m < 2p$. Hence, $w^+ \in L^m (\disk)$, for $0 < m < 2p$. By Proposition \ref{prop: pminlpfuncinlp}, $w \in L^m(\disk, \bc)$, for $0 < m < 2p$. 
\end{proof}

\subsection{Boundary Behavior}

Now, we show that functions in $H^p_{Bel,\mu}(\disk,\bc)$ have nontangential boundary values in $L^p(\p\disk,\bc)$ and converge to those nontangential boundary values in the $L^p$ norm. This generalizes Theorem \ref{bvcon} and Theorem \ref{thm: belbvcon} to the setting of Hardy classes of solutions to $\bc$-Beltrami equations.

    \begin{theorem}\label{thm: bcbelbv}\label{thm: bcbetterbvbetterleb}\label{thm: bczeroonbcimpliesequivzero}
        For $\mu : \disk \to \bc$ such that there exists a real constant $c$ satisfying $||\mu||_{L^\infty(\disk,\bc)} \leq c < 1$ and $0 < p < \infty$, the following three statements follow:
	\begin{enumerate}
		\item Every $w \in H^p_{Bel,\mu}(\disk,\bc)$ has a nontangential boundary value $w_{nt} \in L^p(\p \disk, \bc)$ and
        \[
        \lim_{r \nearrow 1} \int_0^{2\pi} ||w_{nt}(e^{i\theta}) - w(re^{i\theta}) ||^p_{\bc} \, d\theta  = 0. 
    \]  
	\item Every $w \in H^p_{Bel,\mu}(\disk,\bc)$ with nontangential boundary value $w_{nt} \in L^s(\p\disk,\bc)$, where $s>p$, is an element of $H^s_{Bel,\mu}(\disk,\bc)$. 
	\item Every $w \in H^p_{Bel,\mu}(\disk,\bc)$ with nontangential boundary value $w_{nt}$ that vanishes on a set $E \subset \p \disk$ of positive measure is identically equal to zero. 
	\end{enumerate}
    \end{theorem}

    \begin{proof}
        By Theorem \ref{thm: bcbelhardyrep}, $w \in H^p_{Bel,\mu}(\disk,\bc)$ if and only if $(w^+)^* \in H^p_{Bel, (\mu^+)^*}(\disk)$ and $w^- \in H^p_{Bel,\mu^-}(\disk)$. By Theorem \ref{thm: belbvcon}, $(w^+)^*$ and $w^-$ have nontangential boundary values $(w^+)^*_{nt}$ and $w^-_{nt}$, respectively, $(w^+)^*_{nt}, w^-_{nt} \in L^p(\p \disk)$,
        \begin{equation}\label{eqn: plusgoestozero}
                \lim_{r\nearrow 1} \int_0^{2\pi} |(w^+)^*(re^{i\theta})- (w^+)^*_{nt}(e^{i\theta})|^p \, d\theta = 0,
        \end{equation}
        and 
        \begin{equation}\label{eqn: minusgoestozero}
            \lim_{r\nearrow 1} \int_0^{2\pi} |w^-(re^{i\theta})- w^-_{nt}(e^{i\theta})|^p \, d\theta = 0.
        \end{equation}
        Since $w = p^+ w^+ + p^- w^-$, it follows that $w_{nt} = p^+ w^+_{nt} + p^- w^-_{nt}$ exists, where $w^+_{nt} = ((w^+)^*_{nt})^*$. Since $|w^+_{nt}(e^{i\theta})| = |(w^+)^*_{nt}(e^{i\theta})|$, for all $e^{i\theta} \in \p \disk$, it follows that $w^+_{nt} \in L^p(\p \disk)$. Thus, by \eqref{bcbasicestimates}, we have
        \begin{align*}
            \int_0^{2\pi} ||w_{nt}(e^{i\theta})||_{\bc}^p\,d\theta 
            &\leq 2^{-p/2}\int_0^{2\pi} \left( |w^+_{nt}(e^{i\theta})| + |w^-_{nt}(e^{i\theta})|\right)^p\,d\theta\\
            &\leq C_p \left(\int_0^{2\pi}  |w^+_{nt}(e^{i\theta})|^p \,d\theta + \int_0^{2\pi}  |w^-_{nt}(e^{i\theta})|^p\,d\theta  \right) < \infty,
        \end{align*}
        where $C_p$ is a constant that depends on only $p$, and $w_{nt} \in L^p(\p\disk,\bc)$. 

        Now, for $r \in (0,1)$, observe that
        \begin{align*}
            & \int_0^{2\pi} ||w_{nt}(e^{i\theta}) - w(re^{i\theta}) ||^p_{\bc} \, d\theta\\
            &\leq C_p \left( \int_0^{2\pi} |w^+_{nt}(e^{i\theta}) - w^+(re^{i\theta})|^p \,d\theta +  \int_0^{2\pi} |w^-_{nt}(e^{i\theta}) - w^-(re^{i\theta})|^p \,d\theta \right),
        \end{align*}
        where $C_p$ is a constant that depends only on $p$, by \eqref{bcbasicestimates}. Therefore, by \eqref{eqn: plusgoestozero} and \eqref{eqn: minusgoestozero}, 
        \begin{align*}
            &\lim_{r\nearrow 1} \int_0^{2\pi} ||w_{nt}(e^{i\theta}) - w(re^{i\theta}) ||^p_{\bc} \, d\theta \\
            &\leq \lim_{r \nearrow 1} C_p \left( \int_0^{2\pi} |w^+_{nt}(e^{i\theta}) - w^+(re^{i\theta})|^p \,d\theta +  \int_0^{2\pi} |w^-_{nt}(e^{i\theta}) - w^-(re^{i\theta})|^p \,d\theta \right)= 0.
        \end{align*}

        Now, suppose $w_{nt} \in L^s(\p\disk,\bc)$. Since $w_{nt} \in L^s(\p\disk,\bc)$, it follows by Propositions \ref{everybchasplusandminus} and  \ref{prop: pminlpfuncinlp} that $w_{nt} = p^+w^+_{nt} + p^-w^-_{nt}$ and $w^\pm_{nt} \in L^s(\p \disk)$.  Hence, $(w^+_{nt})^* \in L^s(\p\disk)$. By Theorem \ref{thm: betterbvbetterleb}, since $(w^+_{nt})^*, w^-_{nt} \in L^s(\p \disk)$, it follows that $(w^+)^* \in H^s_{Bel,(\mu^+)^*}(\disk)$ and $w^- \in H^s_{Bel,\mu^-}(\disk)$. By Theorem \ref{thm: bcbelhardyrep}, since $(w^+)^* \in H^s_{Bel,(\mu^+)^*}(\disk)$ and $w^- \in H^s_{Bel,\mu^-}(\disk)$, it follows that $w \in H^s_{Bel,\mu}(\disk,\bc)$.

        Suppose $w_{nt} \equiv 0$ on $E$. So,
\[
   0 \equiv  w_{nt} = p^+ w^+_{nt} + p^- w^-_{nt} .
\]
Since $p^+$ and $p^-$ are linearly independent, it follows that 
\[
        w^+_{nt} \equiv 0 \quad\text{and}\quad w^-_{nt} \equiv 0
\]
on $E$. Since $|(w^+)_{nt}^*(z)| = |w^+_{nt}(z)|$, for all $z \in \p \disk$, it follows that 
\[
        (w^+)^*_{nt} \equiv 0
\]
on $E$. By Theorem \ref{thm: zeroposmesidentzero}, we have $(w^+)^* \equiv 0 \equiv w^-$. Therefore, $w \equiv 0$. 

    \end{proof}

\subsection{Higher-Order Iterated Beltrami Equation}

Next, we define a natural higher-order generalization of the Beltrami equation and analyze its solutions. This generalization is realized by iterating the differential operator associated with the first-order equation. Constructing a partial differential equation by iteration of an operator associated with a well-studied equation was previously considered in the setting of polyanalytic functions (see \cite{Balk} and many others) where the differential operator is the classic Cauchy-Riemann operator associated with holomorphic functions and the setting of Vekua equations in the complex setting in \cite{itvek, itvekbvp, Pascali, metahardy, WB3, WBD} and the bicomplex setting in \cite{BCHoiv}. Since this has not been thoroughly studied for Beltrami equations (even in the complex case), we consider it here and begin in the complex setting. 

\begin{deff}
     Let $n$ be a positive integer, $\mu \in W^{n-1,\infty}(\disk)$ such that $||\mu||_{W^{n-1,\infty}(\disk)}\leq c < 1$ for some positive real number $c$, and $w : \disk\to\C$. A \textit{higher-order iterated Beltrami equation} (abbreviated HOIB-equation) is any equation of the form
     \[
        \left( \frac{\p }{\p z^*} - \mu\frac{\p}{\p z} \right)^n w  = 0.
     \]
\end{deff}

The next representation theorem extends Theorem 2.2 of \cite{GuoJinDangDu} (see also Theorem 2.3 of \cite{GuoZhang}) where the same conclusion was proved for a restricted class of higher-order iterated Beltrami equations. Our proof uses the same argument that justifies representation formulas for polyanalytic functions in \cite{Balk} and solutions to higher-order iterated Vekua equation in \cite{WBD}. 

\begin{theorem}\label{thm: hoibimpliesbel}
    Let $n$ be a positive integer and $\mu \in W^{n-1,\infty}(\disk)$ such that $||\mu||_{W^{n-1,\infty}(\disk)}\leq c < 1$ for some positive real number $c$. Every solution $w : \disk \to\C$ of 
    \begin{equation}\label{higherbeltramiequation}
        \left( \frac{\p }{\p z^*} - \mu\frac{\p}{\p z} \right)^n w  = 0
    \end{equation}
    has the form 
    \begin{equation}\label{higherbeltramirep}
        w = \sum_{k=0}^{n-1}  (z^*)^k w_k,
    \end{equation}
    where $\frac{\p w_k}{\p  z^*} = \mu \frac{\p w_k}{\p z}$, for every $0 \leq k \leq n-1$. 
\end{theorem}

\begin{proof}
    Functions of the form of \eqref{higherbeltramirep} are solutions of \eqref{higherbeltramiequation} by direct computation. 

    For the converse, the case $n = 1$ is clear. Now, suppose the representation holds for all $n$ such that $1 \leq n \leq m-1$. Let $w$ be any solution of 
    \[
        \left( \frac{\p }{\p z^*} - \mu\frac{\p}{\p z} \right)^n w  = 0
    \]
    Hence, 
    \[
        \gamma:= \left( \frac{\p }{\p z^*} - \mu\frac{\p}{\p z} \right) w
    \]
    is a solution of 
    \[
        \left( \frac{\p }{\p z^*} - \mu\frac{\p}{\p z} \right)^{n-1} \gamma = 0.
    \]
    By the induction hypothesis, there exist $\{w_k\}_{k=0}^{m-2}$ such that 
    \[
        \frac{\p w_k}{\p  z^*} = \mu \frac{\p w_k}{\p z},
    \]
    for every $0 \leq k \leq m-2$, and 
    \[
        \gamma = \sum_{k=0}^{m-2}  (z^*)^k w_k.
    \]
    Observe that 
    \begin{align*}
        &\left( \frac{\p }{\p z^*} - \mu\frac{\p }{\p z}\right)\left( w - \sum_{k=0}^{m-2} \frac{1}{k+1} (z^*)^{k+1} w_k \right) \\
        &= \sum_{k=0}^{m-2}  (z^*)^k w_k - \sum_{k=0}^{m-2} \left(  (z^*)^k w_k  + \frac{1}{k+1}  (z^*)^{k+1} \frac{\p w_k}{\p z^*} - \mu\frac{1}{k+1} (z^*)^{k+1} \frac{\p w_k}{\p z}\right)\\
        &= 0.
    \end{align*}
    Thus, 
    \[
        w = \sigma +  \sum_{k=0}^{m-2} \frac{1}{k+1} (z^*)^{k+1} w_k,
    \]
    where $\frac{\p \sigma}{\p  z^*} = \mu \frac{\p \sigma}{\p z}$. Therefore, letting $\tilde{w}_0 = \sigma$ and $\tilde{w}_{k} = w_{k+1}$, for $0 \leq k \leq m-2$, we have
    \[
        w = \sum_{k=0}^{m-1}  (z^*)^k \tilde{w}_k,
    \]
    where $\frac{\p \tilde{w}_k}{\p  z^*} = \mu \frac{\p \tilde{w}_k}{\p z}$, for every $0 \leq k \leq m-1$. 
\end{proof}

Now, we define Hardy classes of solutions to higher-order iterated Beltrami equations.

\begin{deff}
    For $p$ a positive real number, $n$ a positive integer, and $\mu \in W^{n-1,\infty}(\disk)$ such that $||\mu||_{W^{n-1,\infty}(\disk)}\leq c < 1$ for some positive real number $c$, we define the HOIB-Hardy classes $H^{n,p}_{Bel,\mu}(\disk)$ to be the collection of functions $w:\disk\to\C$ that solve
    \[
        \left( \frac{\p }{\p z^*} - \mu\frac{\p}{\p z} \right)^n w  = 0
    \]
    and satisfy
    \[
        \sum_{k=0}^{n-1} \sup_{0 < r< 1}\int_0^{2\pi} \left|\left(\frac{\p}{\p  z^*} - \mu\frac{\p}{\p z}\right)^k w(re^{i\theta})\right|^p\,d\theta < \infty.
    \]
\end{deff}

\begin{comm}
Note, for $\mu \equiv 0$, $H^{n,p}_{Bel,\mu}(\disk)$ reduces to the poly-Hardy classes of functions. See \cite{polyhardy}. 
\end{comm}

    The motivation for the definition above becomes clear with the realization that if $w = \sum_{k=0}^{n-1}  (z^*)^k w_k$, where each $w_k$ solves the classic first-order Beltrami equation, then we can represent each $w_k$ by 
    \begin{equation}\label{eqn: componentrep}
        w_k = \frac{1}{k!} \sum_{j=0}^{n-1-k}\frac{(-1)^j}{j!}  (z^*)^j \left(\frac{\p}{\p z^*} - \mu\frac{\p}{\p z}\right)^{k+j} w,
    \end{equation}
    for $0 \leq k \leq n-1$. A similar representation holds for the components of solutions to higher-order iterated Vekua equations (see \cite{itvek, WB3,WBD}). Following the argument from Theorem 4.1 of \cite{WBD}, we use this representation to show the following relationship between Hardy classes of solutions to first-order Beltrami equations and the associated Hardy classes of solutions to higher-order equations Beltrami equations.

\begin{theorem}\label{thm: higherBelHardyimpliescomponentsinBelHardy}
For $p$ a positive real number, $n$ a positive integer, and $\mu \in W^{n-1,\infty}(\disk)$ such that $||\mu||_{W^{n-1,\infty}(\disk)}\leq c < 1$ for some positive real number $c$, $w = \sum_{k=0}^{n-1}  (z^*)^k w_k \in H^{n,p}_{Bel,\mu}(\disk)$ if and only if $w_k \in H^p_{Bel,\mu}(\disk)$, for $0 \leq k \leq n-1$. 
\end{theorem}

\begin{proof}

First, suppose that $w = \sum_{k=0}^{n-1}(z^*)^k w_k$, and $w_k \in H^{p}_{Bel,\mu}(\disk)$. So, 
\[
     \left(\frac{\p }{\p z^*}- \mu\frac{\p}{\p z}\right)^k  w = \sum_{j = k}^{n-1} C_{k,j} (z^*)^{j-k} w_j,
\]
for every $0 \leq k \leq n-1$, where the $C_{k,j}$ are constants that depend only on $k$ and $j$. Observe that
\begin{align*}
    & \sup_{0 < r <1} \int_0^{2\pi} \left|  \left(\frac{\p }{\p z^*}- \mu\frac{\p}{\p z}\right)^k  w(re^{i\theta})  \right|^p \,d\theta \\
    &= \sup_{0 < r <1} \int_0^{2\pi} \left| \sum_{j = k}^{n-1} C_{k,j} (re^{-i\theta})^{j-k} w_j(re^{i\theta})  \right|^p \,d\theta \\
    &\leq C \sum_{j = k}^{n-1} \sup_{0 < r < 1 } \int_0^{2\pi} \left|w_j(re^{i\theta})  \right|^p \,d\theta < \infty,
\end{align*}
where $C$ is a constant that depends only on $n$ and $p$. Therefore, $w \in H_{Bel,\mu}^{n,p}(\disk)$.

Now, suppose that $w \in H^{n,p}_{Bel,\mu}(\disk)$. By \eqref{eqn: componentrep}, we have
\begin{align*}
    &\sup_{0 < r <1} \int_0^{2\pi} \left| w_k(re^{i\theta})  \right|^p \,d\theta \\
    &= \sup_{0 < r <1} \int_0^{2\pi} \left|\sum_{j=0}^{n-1-k}C_{k,j}(re^{-i\theta})^j\, \left(\frac{\p}{\p z^*}-\mu\frac{\p}{\p z}\right)^{k+j}w(re^{i\theta})  \right|^p \,d\theta \\
    &\leq C \sum_{j=0}^{n-1-k}\sup_{0 < r < 1} \int_0^{2\pi} \left| \left(\frac{\p}{\p z^*}-\mu\frac{\p}{\p z}\right)^{k+j}w(re^{i\theta})  \right|^p \,d\theta < \infty,
\end{align*}
 where the $C_{k,j}$ are constants that depend on only $k$ and $j$ and $C$ is a constant that depends on only $n$ and $p$, for any $0 \leq k \leq n-1$. Thus, $w_k \in H^p_{Bel,\mu}(\disk)$.

\end{proof}

\begin{theorem}For $p$ a positive real number, $n$ a positive integer, and $\mu \in W^{n-1,\infty}(\disk)$ such that $||\mu||_{W^{n-1,\infty}(\disk)}\leq c < 1$ for some positive real number $c$, every $w \in H^{n,p}_{Bel,\mu}(\disk)$ has a nontangential boundary value $w_{nt} \in L^p(\p \disk)$ and 
\[
    \lim_{r\nearrow 1} \int_0^{2\pi} |w(re^{i\theta})- w_{nt}(e^{i\theta})|^p\,d\theta = 0.
\]

\end{theorem}

\begin{proof}

    By Theorem \ref{thm: higherBelHardyimpliescomponentsinBelHardy}, every $w \in H^{n,p}_{Bel,\mu}(\disk)$ is representable as 
    \[
        w = \sum_{k=0}^{n-1} (z^*)^k w_k,
    \]
    where $w_k \in H^p_{Bel,\mu}(\disk)$, for all $k$. By Theorem \ref{thm: belbvcon}, every $w_k$ has a nontangential boundary value $w_{k,nt} \in L^p(\p \disk)$ and 
    \[
        \lim_{r\nearrow 1} \int_0^{2\pi} |w_k(re^{i\theta})- w_{k,nt}(e^{i\theta})|^p\,d\theta = 0.
    \]
    Since $w = \sum_{k=0}^{n-1} (z^*)^k w_k$, it follows that $w_{nt} = \sum_{k=0}^{n-1} (z^*)^k w_{k,nt}$ exists and is in $L^p(\p\disk)$. Now, for $r \in (0,1)$, observe that 
    \begin{align*}
        &  \int_0^{2\pi} |w(re^{i\theta})- w_{nt}(e^{i\theta})|^p\,d\theta\\
        &=  \int_0^{2\pi} \left|\sum_{k=0}^{n-1} r^k e^{-ik\theta} w_k(re^{i\theta})- \sum_{k=0}^{n-1} e^{-ik\theta} w_{k,nt}(e^{i\theta}) \right|^p\,d\theta\\
        &\leq C_{p,n} \sum_{k=0}^{n-1} \int_0^{2\pi} |r^k w_k(re^{i\theta}) - w_{k,nt}(e^{i\theta})|^p \,d\theta\\
        &= C_{p,n} \sum_{k=0}^{n-1} \int_0^{2\pi} |r^k w_k(re^{i\theta}) - w_k(re^{i\theta}) + w_k(re^{i\theta})- w_{k,nt}(e^{i\theta})|^p \,d\theta\\
        &\leq C_{p,n} \sum_{k=0}^{n-1} \left( \int_0^{2\pi} |r^k w_k(re^{i\theta}) - w_k(re^{i\theta})|^p \,d\theta+ \int_0^{2\pi} |w_k(re^{i\theta})- w_{k,nt}(e^{i\theta})|^p \,d\theta \right),
    \end{align*}
    where $C_{p,n}$ is a constant that depends only on $p$ and $n$ and is not necessarily the same from line to line. Therefore, by Lebesgue's Dominated Convergence Theorem, 
    \begin{align*}
       & \lim_{r\nearrow 1} \int_0^{2\pi} |w(re^{i\theta})- w_{nt}(e^{i\theta})|^p\,d\theta \\
       & \leq C_{p,n} \sum_{k=0}^{n-1} \left( \lim_{r\nearrow 1} \int_0^{2\pi} |r^k-1|^p \,|w_k(re^{i\theta}) |^p \,d\theta+ \lim_{r\nearrow 1}\int_0^{2\pi} |w_k(re^{i\theta})- w_{k,nt}(e^{i\theta})|^p \,d\theta \right) \\
       &= 0.
    \end{align*}

\end{proof}

Now, we extend the preceding ideas into the setting of bicomplex numbers. Solutions of differential equations constructed by iteration of the differential operator were previously considered in the bicomplex setting in \cite{BCAtomic, BCHoiv}.

\begin{deff}
Let $n$ be a positive integer, $\mu \in W^{n-1,\infty}(\disk,\bc)$ such that $||\mu||_{W^{n-1,\infty}(\disk,\bc)}\leq c < 1$ for some positive real number $c$, and $w : \disk\to\bc$. A \textit{bicomplex higher-order iterated Beltrami equation} (abbreviated $\bc$-HOIB-equation) is any equation of the form
     \[
        \left( \dbar - \mu\p \right)^n w  = 0.
     \]
\end{deff}

The next representation extends Theorem \ref{thm: hoibimpliesbel} from the complex setting (that also holds for solutions of higher-order iterated Vekua-type equations \cite{WBD, BCHoiv}) that allows us to represent the solution of the higher-order equation in terms of a finite sum of solutions to the first-order equation. With this representation, we prove properties of the solutions to the first-order equation extend to solutions of the higher-order equation.

\begin{theorem}\label{thm: hoibrep}
 Let $n$ be a positive integer and $\mu \in W^{n-1,\infty}(\disk,\bc)$ such that $||\mu||_{W^{n-1,\infty}(\disk,\bc)}\leq c < 1$ for some positive real number $c$. Every solution $w : \disk \to\bc$ of 
    \begin{equation}\label{bchigherbeltramiequation}
        \left( \dbar - \mu\p \right)^n w  = 0
    \end{equation}
    has the form 
    \begin{equation}\label{bchigherbeltramiequationrepresentation}
        w = \sum_{k=0}^{n-1} \widehat{ (z^*)}^k w_k,
    \end{equation}
    where $\dbar w_k = \mu \p w_k$, for every $k$. 
\end{theorem}

\begin{proof}
Observe that any function of the form $w = \sum_{k=0}^{n-1} \widehat{(z^*)}^k w_k$, where $\dbar w_k = \mu \p w_k$ for every $k$, is a solution of \eqref{bchigherbeltramiequation} by direct computation. 

For the other direction, we proceed by induction. The $n =1$ case is trivial. Suppose that for every $0 \leq n \leq m-1$, if $w$ is a solution of $(\dbar - \mu\p)^n w = 0$, then $w = \sum_{k=0}^{n-1} \widehat{(z^*)}^k w_k$, where $\dbar w_k = \mu \p w_k$ for every $k$. Let $f$ be a solution of $(\dbar - \mu \p)^m f = 0$. So, $ g = (\dbar - \mu \p)f$ solves $(\dbar - \mu\p)^{m-1}g = 0$ and, by the induction hypothesis, must have the form
\[
    (\dbar - \mu \p)f = \sum_{k=0}^{m-2
    }\widehat{(z^*)}^kf_k,
\]
where $\dbar f_k = \mu \p f_k$, for every $k$. Observe that 
\begin{align*}
    & (\dbar - \mu \p)\left( f - \sum_{k=0}^{m-2}\frac{1}{k+1} \widehat{(z^*)}^{k+1} f_k\right)\\
    &= \sum_{k=0}^{m-2
    }\widehat{(z^*)}^kf_k - (\dbar - \mu \p)\sum_{k=0}^{m-2}\frac{1}{k+1} \widehat{(z^*)}^{k+1} f_k \\
    &= \sum_{k=0}^{m-2
    }\widehat{(z^*)}^kf_k - \sum_{k=0}^{m-2
    }\widehat{(z^*)}^kf_k  = 0.
\end{align*}
Hence, 
\[
    f - \sum_{k=0}^{m-2}\frac{1}{k+1} \widehat{(z^*)}^{k+1} f_k = \phi,
\]
where $\dbar \phi = \mu \p \phi$. Thus, 
\[
     f = \sum_{k=0}^{m-1} \widehat{(z^*)}^k \tilde{f_k}, 
\]
where $\tilde{f}_{k+1} := \frac{1}{k+1}f_{k}$, $0 \leq k \leq m-2$ and $\tilde{f_0} = \phi$. Therefore, the representation \eqref{bchigherbeltramiequationrepresentation} holds for all orders. 

\end{proof}

\begin{comm}
    Recall that solutions of \eqref{higherbeltramiequation} are representable as \eqref{higherbeltramirep}, and that the components functions of that representation can be represented by \eqref{eqn: componentrep}. A similar situation is present for solutions of 
    \[
        \left( \dbar - \mu\p \right)^n w  = 0
    \]
    with representation 
    \[
        w = \sum_{k=0}^{n-1} \widehat{ (z^*)}^k w_k.
    \]
    The component functions are representable by the formula
    \begin{equation}\label{eqn: bccomponentrep}
         w_k = \frac{1}{k!} \sum_{j=0}^{n-1-k}\frac{(-1)^j}{j!}  \widehat{(z^*)}^j \left(\dbar - \mu\p\right)^{k+j} w.
    \end{equation}
\end{comm}

Now, we define Hardy classes of solutions to bicomplex higher-order iterated Beltrami equations.

\begin{deff}
    For $p$ a positive real number, $n$ a positive integer, and $\mu \in W^{n-1,\infty}(\disk,\bc)$ such that $||\mu||_{W^{n-1,\infty}(\disk,\bc)}\leq c < 1$ for some positive real number $c$, we define the $\bc$-HOIB-Hardy classes $H^{n,p}_{Bel,\mu}(\disk,\bc)$ to be the collection of functions $w:\disk\to\bc$ that solve
    \[
     \left( \dbar - \mu\p \right)^n w  = 0
     \]
    and satisfy
    \[
        \sum_{k=0}^{n-1} \sup_{0 < r< 1}\int_0^{2\pi} \left|\left|\left( \dbar - \mu\p \right)^k w(re^{i\theta})\right|\right|_\bc^p\,d\theta < \infty.
    \]
\end{deff}

With the definition above, we immediately extend Theorem \ref{thm: higherBelHardyimpliescomponentsinBelHardy} to $H^{n,p}_{Bel,\mu}(\disk,\bc)$ in the following way.

\begin{theorem}\label{thm: bchoibinclusion}
For $p$ a positive real number, $n$ a positive integer, and $\mu \in W^{n-1,\infty}(\disk,\bc)$ such that $||\mu||_{W^{n-1,\infty}(\disk,\bc)}\leq c < 1$ for some positive real number $c$, $w = \sum_{k=0}^{n-1} \widehat{ (z^*)}^k w_k \in H^{n,p}_{Bel,\mu}(\disk,\bc)$ if and only if $w_k \in H^p_{Bel,\mu}(\disk,\bc)$, for $0 \leq k \leq n-1$. 
\end{theorem}

\begin{proof}

First, suppose that $w = \sum_{k=0}^{n-1}\widehat{(z^*)}^k w_k$, and $w_k \in H^{p}_{Bel,\mu}(\disk,\bc)$. So, 
\[
     \left(\dbar- \mu\p\right)^k  w = \sum_{j = k}^{n-1} C_{k,j} \widehat{(z^*)}^{j-k} w_j,
\]
for every $0 \leq k \leq n-1$, where the $C_{k,j}$ are constants that depend only on $k$ and $j$. Observe that, for each $0 \leq k \leq n-1$, 
\begin{align*}
    & \sup_{0 < r <1} \int_0^{2\pi} \left||  (\dbar - \mu \p)^k  w(re^{i\theta})  \right||_{\bc}^p \,d\theta \\
    &= \sup_{0 < r <1} \int_0^{2\pi} \left|\left| \sum_{j = k}^{n-1} C_{k,j} \widehat{(re^{-i\theta})}^{j-k} w_j(re^{i\theta})  \right|\right|_\bc^p \,d\theta \\
    &\leq C \sum_{j = k}^{n-1} \sup_{0 < r < 1 } \int_0^{2\pi} \left|\left|w_j(re^{i\theta})  \right|\right|^p \,d\theta < \infty,
\end{align*}
where $C$ is a constant that depends only on $n$ and $p$. Therefore, $w \in H_{Bel,\mu}^{n,p}(\disk)$.

Now, suppose that $w \in H^{n,p}_{Bel,\mu}(\disk)$. By \eqref{eqn: bccomponentrep}, we have
\begin{align*}
    &\sup_{0 < r <1} \int_0^{2\pi} \left|\left| w_k(re^{i\theta})  \right|\right|_\bc^p \,d\theta \\
    &= \sup_{0 < r <1} \int_0^{2\pi} \left|\left|\sum_{j=0}^{n-1-k}C_{k,j} \widehat{(re^{-i\theta})}^j\, \left(\dbar - \mu \p\right)^{k+j}w(re^{i\theta})  \right|\right|_\bc^p \,d\theta \\
    &\leq C \sum_{j=0}^{n-1-k}\sup_{0 < r < 1} \int_0^{2\pi} \left|\left| \left(\dbar - \mu \p\right)^{k+j}w(re^{i\theta})  \right|\right|_\bc^p \,d\theta < \infty,
\end{align*}
 where the $C_{k,j}$ are constants that depend on only $k$ and $j$ and $C$ is a constant that depends on only $n$ and $p$, for any $0 \leq k \leq n-1$. Thus, $w_k \in H^p_{Bel,\mu}(\disk)$.

\end{proof}

\begin{corr}\label{corr: bcbelhpcompinhp}
    For $p$ a positive real number, $n$ a positive integer, and $\mu \in W^{n-1,\infty}(\disk,\bc)$ such that $||\mu||_{W^{n-1,\infty}(\disk,\bc)}\leq c < 1$ for some positive real number $c$, $w = \sum_{k=0}^{n-1} \widehat{ (z^*)}^k w_k \in H^{n,p}_{Bel,\mu}(\disk,\bc)$ if and only if $(w_k^+)^* \in H^p_{Bel,(\mu^+)^*}(\disk)$ and $w^-_k \in H^p_{Bel,\mu^-}(\disk)$, for $0 \leq k \leq n-1$. 
\end{corr}

\begin{corr}
    For $p$ a positive real number, $n$ a positive integer, and $\mu \in W^{n-1,\infty}(\disk,\bc)$ such that $||\mu||_{W^{n-1,\infty}(\disk,\bc)}\leq c < 1$ for some positive real number $c$, every $w \in H^{n,p}_{Bel,\mu}(\disk,\bc)$ is an element of $L^m(\disk,\bc)$, for every $0 < m < 2p$. 
\end{corr}

\begin{proof}
    By Corollary \ref{corr: bcbelhpcompinhp}, if $w \in H^{n,p}_{Bel,\mu}(\disk,\bc)$, then $w = \sum_{k=0}^{n-1} \widehat{ (z^*)}^k w_k$ and $(w_k^+)^* \in H^p_{Bel,(\mu^+)^*}(\disk)$ and $w^-_k \in H^p_{Bel,\mu^-}(\disk)$, for $0 \leq k \leq n-1$. By Theorem \ref{thm: belhpinbetterlp}, $(w^+_k)^*, w^-_k \in L^m(\disk)$, for $0 < m < 2p$. By Proposition \ref{prop: pminlpfuncinlp}, $w_k \in L^m(\disk, \bc)$, for $0 < m < 2p$ and every 
    $k$. Thus, for any $0 < m < 2p$, we have 
    \begin{align*}
        \iint_{\disk} ||w(z)||_\bc^m \,dx\,dy
        &= \iint_{\disk} \left|\left|\sum_{k=0}^{n-1} \widehat{ (z^*)}^k w_k(z) \right|\right|_\bc^m \,dx\,dy\\
        &\leq C \sum_{k=0}^{n-1} \iint_{\disk} ||w_k(z)||_\bc^m\,dx\,dy < \infty,
    \end{align*}
    where $C$ is a constant that depends only on $n$ and $m$, by \eqref{bcbasicestimates}. Therefore, $w \in L^m(\disk,\bc)$.
\end{proof}

We also recover the classic Hardy space boundary behavior.

\begin{theorem}\label{thm: bcbelbvcon}
For $p$ a positive real number, $n$ a positive integer, and $\mu \in W^{n-1,\infty}(\disk,\bc)$ such that $||\mu||_{W^{n-1,\infty}(\disk,\bc)}\leq c < 1$, for some positive real number $c$, every $w \in H^{n,p}_{Bel,\mu}(\disk,\bc)$ has a nontangential boundary value $w_{nt} \in L^p(\p \disk,\bc)$ and 
\[
    \lim_{r\nearrow 1} \int_0^{2\pi} ||w(re^{i\theta})- w_{nt}(e^{i\theta})||_{\bc}^p\,d\theta = 0.
\]

\end{theorem}

\begin{proof}

 Let $w \in H^{n,p}_{Bel,\mu}(\disk,\bc)$. By Theorem \ref{thm: bchoibinclusion}, 
        \[
            w = \sum_{k = 0}^{n-1}  \widehat{(z^*)}^k w_k,
        \]
        where $w_k \in H^p_{Bel,\mu}(\disk,\bc)$, for every $k$. By Theorem \ref{thm: bcbelbv}, since $w_k \in H^p_{Bel,\mu}(\disk,\bc)$, for every $k$, it follows that $w_{k,nt}$ exists, is in $L^p(\p \disk,\bc)$, and 
        \begin{equation}\label{convfork}
            \lim_{r \nearrow 1} \int_0^{2\pi} ||w_{k,nt}(e^{i\theta}) - w_k(re^{i\theta}) ||^p_{\bc} \, d\theta  = 0. 
        \end{equation}

        Since $w = \sum_{k = 0}^{n-1}  \widehat{(z^*)}^k w_k$, it follows that 
        \[
            w_{nt} = \sum_{k = 0}^{n-1}  \widehat{(e^{i\theta})^*}^k w_{k,nt}
        \]
        exists, and, for $r \in (0,1)$, we have
        \begin{align*}
            &\int_0^{2\pi} ||w_{nt}(e^{i\theta}) - w(re^{i\theta}) ||^p_{\bc} \, d\theta \\
            &= \int_0^{2\pi} ||\sum_{k = 0}^{n-1}  \widehat{(e^{i\theta})^*}^k w_{k,nt}(e^{i\theta}) - \sum_{k = 0}^{n-1}  \widehat{(re^{i\theta})^*}^k w_k(re^{i\theta}) ||^p_{\bc} \, d\theta \\
            &\leq C_p \int_0^{2\pi} ||\sum_{k = 0}^{n-1}  \widehat{(e^{i\theta})^*}^k w_{k,nt}(e^{i\theta}) - \sum_{k = 0}^{n-1}  \widehat{(e^{i\theta})^*}^k w_k(re^{i\theta}) ||^p_{\bc} \, d\theta \\
            &\quad\quad + C_p\int_0^{2\pi} ||\sum_{k = 0}^{n-1}  \widehat{(e^{i\theta})^*}^k w_k(re^{i\theta}) - \sum_{k = 0}^{n-1}  \widehat{(re^{i\theta})^*}^k w_k(re^{i\theta}) ||^p_{\bc} \, d\theta \\
            &\leq C_{p,n} \sum_{k = 0}^{n-1} \int_0^{2\pi} ||  \widehat{(e^{i\theta})^*}^k (w_{k,nt}(e^{i\theta}) - w_k(re^{i\theta})) ||^p_{\bc} \, d\theta \\
            &\quad\quad + C_{p,n} \sum_{k = 0}^{n-1} \int_0^{2\pi}  ||  (\widehat{(e^{i\theta})^*}^k -  \widehat{(re^{i\theta})^*}^k )w_k(re^{i\theta}) ||^p_{\bc} \, d\theta \\
            &\leq D_{p,n} \sum_{k = 0}^{n-1} \int_0^{2\pi} ||  w_{k,nt}(e^{i\theta}) - w_k(re^{i\theta}) ||^p_{\bc} \, d\theta \\
            &\quad\quad + C_{p,n} \sum_{k = 0}^{n-1} \int_0^{2\pi}  ||  (\widehat{(e^{i\theta})^*}^k -  \widehat{(re^{i\theta})^*}^k )w_k(re^{i\theta}) ||^p_{\bc} \, d\theta ,
        \end{align*}
        where $C_p$, $C_{p,n}$, and $D_{p,n}$ are constants that depend on only $p$ or $p$ and $n$, respectively. Since
        \begin{align*}
            & (\widehat{(e^{i\theta})^*}^k -  \widehat{(re^{i\theta})^*}^k )w_k(re^{i\theta}) \\
            &= (e^{-jk\theta} - r^k e^{-jk\theta}) (p^+ w_k^+(re^{i\theta}) + p^- w_k^-(re^{i\theta})) \\
            &= (p^+ e^{ik\theta} + p^- e^{-ik\theta} - r^k (p^+ e^{ik\theta} + p^- e^{-ik\theta})) (p^+ w_k^+(re^{i\theta}) + p^- w_k^-(re^{i\theta})) \\
            &= (p^+ (1-r^k)e^{ik\theta} + p^- (1-r^k)e^{-ik\theta} ) (p^+ w_k^+(re^{i\theta}) + p^- w_k^-(re^{i\theta})) \\
            &= p^+ (1-r^k)e^{ik\theta}w_k^+(re^{i\theta}) + p^- (1-r^k)e^{-ik\theta} w_k^-(re^{i\theta}),
        \end{align*}

       it follows that 
        \begin{align*}
&           D_{p,n} \sum_{k = 0}^{n-1} \int_0^{2\pi}                ||  w_{k,nt}(e^{i\theta}) -                                   w_k(re^{i\theta}) ||^p_{\bc} \, d\theta \\
            &\quad\quad + C_{p,n} \sum_{k = 0}^{n-1} \int_0^{2\pi}  ||  (\widehat{(e^{i\theta})^*}^k -  \widehat{(re^{i\theta})^*}^k )w_k(re^{i\theta}) ||^p_{\bc} \, d\theta\\
            &= D_{p,n} \sum_{k = 0}^{n-1} \int_0^{2\pi} ||  w_{k,nt}(e^{i\theta}) - w_k(re^{i\theta}) ||^p_{\bc} \, d\theta \\
            &\quad\quad + C_{p,n} \sum_{k = 0}^{n-1} \int_0^{2\pi}  || p^+ (1-r^k)e^{ik\theta}w_k^+(re^{i\theta}) + p^- (1-r^k)e^{-ik\theta} w_k^-(re^{i\theta}) ||^p_{\bc} \, d\theta\\
            &\leq D_{p,n} \sum_{k = 0}^{n-1} \int_0^{2\pi} ||  w_{k,nt}(e^{i\theta}) - w_k(re^{i\theta}) ||^p_{\bc} \, d\theta \\
            &\quad\quad + C_{p,n} 2^{-p/2}\sum_{k = 0}^{n-1} \int_0^{2\pi}  \left( |(1-r^k)e^{ik\theta}w_k^+(re^{i\theta})| + |(1-r^k)e^{-ik\theta} w_k^-(re^{i\theta}) | \right)^p \, d\theta\\
            &\leq D_{p,n} \sum_{k = 0}^{n-1} \int_0^{2\pi} ||  w_{k,nt}(e^{i\theta}) - w_k(re^{i\theta}) ||^p_{\bc} \, d\theta \\
            &\quad\quad + E_{p,n}\sum_{k = 0}^{n-1} \int_0^{2\pi}  |(1-r^k)|^p\, |w_k^+(re^{i\theta})|^p\,d\theta \\
            &\quad\quad + E_{p,n}\sum_{k = 0}^{n-1} \int_0^{2\pi} |(1-r^k)|^p \, | w_k^-(re^{i\theta}) |^p  \, d\theta ,
        \end{align*}
        where $C_{p,n}$, $D_{p,n}$, and $E_{p,n}$ are constants that depend on only $p$ and $n$. Now, the first sum in the last line of the inequality converges to zero, as $r \nearrow 1$, because of (\ref{convfork}), while the second and third sums in the last line of the inequality converge to zero by Lebesgue's Dominated Convergence Theorem. Therefore, 
        \[
            \lim_{r \nearrow 1} \int_0^{2\pi} ||w_{nt}(e^{i\theta}) - w(re^{i\theta}) ||^p_{\bc} \, d\theta  = 0. 
        \]

\end{proof}

\begin{theorem}
For $p$ a positive real number, $n$ a positive integer, and $\mu \in W^{n-1,\infty}(\disk,\bc)$ such that $||\mu||_{W^{n-1,\infty}(\disk,\bc)}\leq c < 1$ for some positive real number $c$, if $w = \sum_{k=0}^{n-1} \widehat{ (z^*)}^k w_k \in H^{n,p}_{Bel,\mu}(\disk,\bc)$  and $w_{k,nt} \in L^s(\p\disk,\bc)$, $s> p$, for $0 \leq k \leq n-1$, then $w \in H^{n,s}_{Bel,\mu}(\disk,\bc)$. 
\end{theorem}

\begin{proof}
    By Theorem \ref{thm: bchoibinclusion}, if $w \in H^{n,p}_{Bel,\mu}(\disk,\bc)$, then $w = \sum_{k=0}^{n-1} \widehat{ (z^*)}^k w_k$ and $w_k \in H^p_{Bel,\mu}(\disk,\bc)$, for every $k$. By Theorem \ref{thm: bcbelbvcon}, if $w_k \in H^p_{Bel,\mu}(\disk,\bc)$, for every $k$, then, for every $k$, the nontangential limit $w_{k,nt} \in L^p(\p\disk,\bc)$ exists. By Theorem \ref{thm: bcbetterbvbetterleb}, since $w_k \in H^p_{Bel,\mu}(\disk,\bc)$ and $w_{k,nt} \in L^s(\p\disk,\bc)$, for every $k$, it follows that $w_k \in H^s_{Bel,\mu}(\disk,\bc)$, for every $k$. Therefore, by Theorem \ref{thm: bchoibinclusion}, $w \in H^{n,s}_{Bel,\mu}(\disk,\bc)$.
\end{proof}

\section{The Conjugate Beltrami Equation}\label{section: conjbel}

In this section, we define a bicomplex version of the conjugate Beltrami equation using the bicomplex $\dbar$ operator from Definition \ref{bcdbardef}. Using these equations, we define Hardy classes of bicomplex-valued solutions and show that these functions exhibit properties of the classic Hardy spaces. This extends work from the complex-valued function setting that can be found in \cite{PozHardy, CompOp, moreVekHardy, conjbel}.

\subsection{Definition and Representation}

\begin{deff}
    Let $\mu : \disk \to \bc$ be such that there exists a real constant $c$ and $||\mu||_{L^\infty(\disk,\bc)} \leq c < 1$. We call any equation of the form 
    \[
        \dbar w = \mu \, \dbar \overline{w}
    \]
    a $\bc$-conjugate-Beltrami equation. 
\end{deff}

\begin{prop}\label{prop: bcconjbelimpliesconjbbel}
    Let $\mu : \disk \to \bc$ be such that there exists a real constant $c$ and $||\mu||_{L^\infty(\disk,\bc)} \leq c < 1$. Every solution $w:\disk \to\bc$ of 
    \[
        \dbar w = \mu \, \dbar \overline{w}
    \]
    has the form 
    \[
        w = p^+ w^+ + p^- w^-,
    \]
    where $w^+:\disk\to\mathbb{C}$ solves the $\mathbb{C}$-conjugate-Beltrami equation
    \[
        \frac{\p (w^+)^*}{\p z^*} = (\mu^+)^* \frac{\p w^+}{\p z^*}
    \]
    and $w^- : \disk \to \mathbb{C}$ solves the $\mathbb{C}$-conjugate-Beltrami equation
    \[
        \frac{\p w^-}{\p z^*} = \mu^- \frac{\p (w^-)^*}{\p z^*}.
    \]
\end{prop}

\begin{proof}

    Using the idempotent representation of the functions and the differential operators, if $w$ solves \eqref{bcbeltramieqn}, then, by  Remark \ref{remark: idempotentdiffop}, we have
    \begin{align*}
        \dbar w &= \mu \dbar \overline{w} \\
        \left(p^+\frac{\p}{\p z} + p^- \frac{\p}{\p z^*}\right)(p^+w^+ + p^- w^-) &= (p^+\mu^+ + p^- \mu^-)  \left(p^+\frac{\p}{\p z} + p^- \frac{\p}{\p z^*}\right)(p^+(w^+)^* + p^- (w^-)^*)\\
        p^+ \frac{\p w^+}{\p z} + p^- \frac{\p w^-}{\p z^*} &= p^+ \mu^+ \frac{\p (w^+)^*}{\p z} + p^- \mu^- \frac{\p (w^-)^*}{\p z^*}.
    \end{align*}
    Since idempotent representations are unique, it follows that
    \[
        \frac{\p w^+}{\p z} = \mu^+ \frac{\p (w^+)^*}{\p z}
    \quad \text{and} \quad
        \frac{\p w^-}{\p z^*} = \mu^- \frac{\p (w^-)^*}{\p z^*}.
    \]
    Applying complex conjugation to both sides of the right-hand side equation, we have
    \[
        \frac{\p (w^+)^*}{\p z^*} = (\mu^+)^* \frac{\p w^+}{\p z^*}
    \quad \text{and} \quad
       \frac{\p w^-}{\p z^*} = \mu^- \frac{\p (w^-)^*}{\p z^*}.
    \]
\end{proof}

    \subsection{Hardy Spaces}

\begin{deff}
For $\mu : \disk \to \bc$ such that $||\mu||_{L^\infty(\disk,\bc)} \leq c < 1$, for some real constant $c$, and $0 < p < \infty$, we define the $\bc$-conjugate-Beltrami-Hardy spaces $H^p_{conj,\mu}(\disk,\bc)$ to be those functions $w: \disk\to\bc$ such that $w$ solves
    \[
        \dbar w = \mu \, \dbar \overline{w}
    \]
and 
    \[
        ||w||_{H^p_{\bc}} := \left( \sup_{0< r < 1} \frac{1}{2\pi}\int_0^{2\pi} ||w(re^{i\theta})||_{\bc}^p\,d\theta\right)^{1/p} < \infty.
    \]
\end{deff}

\begin{theorem}\label{thm: bcconjbelhardyrep}
    For $\mu : \disk \to \bc$ such that there exists a real constant $c$ satisfying $||\mu||_{L^\infty(\disk,\bc)} \leq c < 1$ and $0 < p < \infty$, $w = p^+ w^+ + p^- w^- \in H^p_{conj,\mu}(\disk,\bc)$ if and only if $(w^+)^* \in H^p_{conj, (\mu^+)^*}(\disk)$ and $w^- \in H^p_{conj,\mu^-}(\disk)$. 
\end{theorem}

\begin{proof} By Proposition \ref{prop: bcconjbelimpliesconjbbel}, if $w$ solves $\dbar w = \mu \dbar \overline{w}$, then
    \[
        \frac{\p (w^+)}{\p z^*} = (\mu^+)^* \frac{\p w^+}{\p z^*}
    \quad \text{and} \quad
        \frac{\p w^-}{\p z^*} = \mu^- \frac{\p (w^-)^*}{\p z^*}.
    \]
    By \eqref{bcbasicestimates}, for every $r \in (0,1)$, we have 
    \begin{align*}
         \int_0^{2\pi} |(w^+)^*(re^{i\theta})|^p \,d\theta
         &=
        \int_0^{2\pi} |w^+(re^{i\theta})|^p \,d\theta\\
        &\leq 2^{p/2}\int_0^{2\pi} ||w(re^{i\theta})||_\bc^p \,d\theta \\
        &\leq 2^{p/2} \sup_{0 < r< 1} \int_0^{2\pi} ||w(re^{i\theta})||_\bc^p \,d\theta< \infty
    \end{align*}
    and 
    \begin{align*}
          \int_0^{2\pi} |w^-(re^{i\theta})|^p \,d\theta 
         &\leq 2^{p/2}\int_0^{2\pi} ||w(re^{i\theta})||_\bc^p \,d\theta \\
         &\leq 2^{p/2}\sup_{0 < r < 1}\int_0^{2\pi} ||w(re^{i\theta})||_\bc^p \,d\theta < \infty.
    \end{align*}
    Thus, 
    \[
        \sup_{0 < r < 1}  \int_0^{2\pi} |(w^+)^*(re^{i\theta})|^p \,d\theta < \infty,
    \]
    and
    \[
        \sup_{0 < r < 1}\int_0^{2\pi} |w^-(re^{i\theta})|^p \,d\theta < \infty.
    \]
    Therefore, $(w^+)^* \in H^p_{conj, (\mu^+)^*}(\disk)$ and $w^- \in H^p_{conj,\mu^-}(\disk)$.

    Now, if $(w^+)^* \in H^p_{conj, (\mu^+)^*}(\disk)$ and $w^- \in H^p_{conj,\mu^-}(\disk)$, then $w = p^+ w^+ + p^- w^-$ solves $\dbar w = \mu \dbar \overline{w}$ by direct computation. Also, by \eqref{bcbasicestimates}, we have for every $r \in (0,1)$ that 
    \begin{align*}
        \int_0^{2\pi} ||w(re^{i\theta})||_\bc^p \,d\theta 
        &\leq 2^{-p/2} \int_0^{2\pi} (|w^+(re^{i\theta})|+|w^-(re^{i\theta})|)^p\,d\theta \\
        &\leq C_p \left( \int_0^{2\pi} |(w^+)^*(re^{i\theta})|^p \,d\theta + \int_0^{2\pi} |w^-(re^{i\theta})|^p\,d\theta\right) \\
        &\leq C_p \left( \sup_{0 < r < 1}\int_0^{2\pi} |(w^+)^*(re^{i\theta})|^p \,d\theta + \sup_{0 < r < 1}\int_0^{2\pi} |w^-(re^{i\theta})|^p\,d\theta\right) < \infty,
    \end{align*}
    where $C_p$ is a constant that depends on only $p$. Thus, 
    \[
        \sup_{0 < r < 1}\int_0^{2\pi} ||w(re^{i\theta})||_\bc^p \,d\theta < \infty,
    \]
    and $w \in H^p_{conj,\mu}(\disk,\bc)$. 
\end{proof}

Note that the proof of this result uses the same argument as the proof of Theorem \ref{thm: bcbelhardyrep}. This is true for many of the proofs in this section and the one that follows. In the cases where there are few differences with a previous proof, we indicate the relevant changes but do not repeat all of the details.

\begin{theorem}
    For $\mu \in W^{1,s}(\disk,\bc)$, $s>2$, such that there exists a real constant $c$ and $||\mu||_{L^\infty(\disk,\bc)} \leq c < 1$ and $0 < p < \infty$, every $w \in H^p_{conj,\mu}(\disk,\bc)$ is an element of $L^m(\disk,\bc)$, for every $0 < m < 2p$. 
\end{theorem}

\begin{proof}
The proof of this result is the same as Theorem \ref{thm: bcbelhpinbetterlp} except appeals are made to Theorem \ref{thm: conjbelinclusioninlebesgue} instead of Theorem \ref{thm: belhpinbetterlp} and Theorem \ref{thm: bcconjbelhardyrep} instead of Theorem \ref{thm: bcbelhardyrep}.

\end{proof}

    \subsection{Boundary Behavior}
    
     \begin{theorem}\label{thm: bcconjbelbv}\label{thm: bcconjbetterbvbetterleb}\label{thm: bcconjzerobdimpliesequivzero}
        For $\mu \in W^{1,q}(\disk,\bc)$ such that there exists a real constant $c$ satisfying $||\mu||_{L^\infty(\disk,\bc)} \leq c < 1$ and $0 < p < \infty$, the following three statements follow:
        	\begin{enumerate}
		\item Every $w \in H^p_{conj,\mu}(\disk,\bc)$ has a nontangential boundary value $w_{nt} \in L^p(\p \disk, \bc)$ and
        \[
        \lim_{r \nearrow 1} \int_0^{2\pi} ||w_{nt}(e^{i\theta}) - w(re^{i\theta}) ||^p_{\bc} \, d\theta  = 0. 
    \]  
	\item Every $w \in H^p_{conj,\mu}(\disk,\bc)$ with nontangential boundary value $w_{nt} \in L^s(\p\disk,\bc)$, where $s>p$, is an element of $H^s_{conj,\mu}(\disk,\bc)$
	\item Every $w \in H^p_{conj,\mu}(\disk,\bc)$ with nontangential boundary value $w_{nt}$ that vanishes on a set $E \subset \p \disk$ of positive measure is identically equal to zero. 
	\end{enumerate}
    \end{theorem}

    \begin{proof}

    The proof of this theorem is the same as Theorem \ref{thm: bcbelbv} with an appeal to Theorem \ref{thm: bcconjbelhardyrep} instead of Theorem \ref{thm: bcbelhardyrep} and an appeal to Theorem \ref{conjbvcon} instead of Theorem \ref{thm: belbvcon}.
    
    \end{proof}

\subsection{Connections to Bicomplex Vekua Equations} 

Next, we show the connection between solutions of complex conjugate Beltrami equations and solutions of complex Vekua equation, see Section \ref{subsubsection: conjbelvekconnect}, is maintained in the setting of bicomplex-valued solutions. For more information about solutions of the bicomplex Vekua equation, see, for example,  \cite{BCHoiv, BCTransmutation, BCBergman, BCHarmVek, CastaKrav, FundBicomplex, ComplexSchr}. The next two results generalize aspects of Proposition 3.2.3.1 in \cite{conjbel}.

\begin{theorem}\label{thm: bcconjbelimpliesbcvek}
For real-valued $\mu\in W^{1,\infty}(\disk)$ such that there exists a constant $c$ satisfying $||\mu||_{L^\infty(\disk)}\leq c < 1$, define $\alpha \in L^\infty(\disk)$ by
\[
    \alpha := - \frac{1}{1-\mu^2} \dbar \mu.
\]
The function $f: \disk\to\bc$ solves the $\bc$-conjugate Beltrami equation
\[
        \dbar f = \mu\dbar \,\overline{f}
\]
if and only if the function $w: \disk\to\bc$ defined by 
\[
        w:= \frac{1}{\sqrt{1-\mu^2}} (f-\mu\overline{f})
\]
solves the $\bc$-Vekua equation
\[
        \dbar w = \alpha \overline{w}.
\]
\end{theorem}

\begin{proof}
By Proposition \ref{prop: bcconjbelimpliesconjbbel}, if $f$ solves 
\[
        \dbar f = \mu\dbar \,\overline{f},
\]
then $f^+$ and $f^-$ solve
\[
        \frac{\p (f^+)^*}{\p z^*} = (\mu^+)^* \frac{\p f^+}{\p z^*}
    \]
    and
    \[
        \frac{\p f^-}{\p z^*} = \mu^- \frac{\p (f^-)^*}{\p z^*},
    \]
respectively. Since $\mu$ is real-valued, this means $f^+$ and $f^-$ solve
\[
        \frac{\p (f^+)^*}{\p z^*} = \mu \frac{\p f^+}{\p z^*}
    \]
    and
    \[
        \frac{\p f^-}{\p z^*} = \mu \frac{\p (f^-)^*}{\p z^*},
    \]
respectively.
By Proposition 3.2.3.1 of \cite{conjbel} (see discussion in Section \ref{subsubsection: conjbelvekconnect}), this implies that the function $g$ defined by 
\[
    g := \frac{1}{\sqrt{1-\mu^2}} ((f^+)^* - \mu f^+)
\]
solves the Vekua equation 
\[
    \frac{\p g}{\p z^*} = \alpha g^*
\]
and the function $h$ defined by
\[
    h := \frac{1}{\sqrt{1-\mu^2}} (f^- - \mu (f^-)^*)
\]
solves
\[
    \frac{\p h}{\p z^*} = \alpha h^*.
\]
By Theorem \ref{bcvekimpliescvek}, $w:= p^+ g^* + p^- h$ solves
\[
    \dbar w = \alpha  \overline{w}.
\]
Observe that 
\begin{align*}
        w &= p^+ g^* + p^- h\\
        &= p^+ \left(\frac{1}{\sqrt{1-\mu^2}} ((f^+)^* - \mu f^+)\right)^* + p^- \left( \frac{1}{\sqrt{1-\mu^2}} (f^- - \mu (f^-)^*)\right)\\
        &= p^+ \left(\frac{1}{\sqrt{1-\mu^2}} (f^+ - \mu (f^+)^*)\right) + p^- \left( \frac{1}{\sqrt{1-\mu^2}} (f^- - \mu (f^-)^*)\right)\\
        &= \frac{1}{\sqrt{1-\mu^2}}\left( p^+ f^+ + p^- f^- -\mu(p^+ (f^+)^* + p^- (f^-)^*) \right)\\
        &= \frac{1}{\sqrt{1-\mu^2}} ( f- \mu \overline{f}).
\end{align*}

Now, suppose that $w$ solves $\bc$-Vekua equation
\[
        \dbar w = \alpha \overline{w}.
\]
By Theorem \ref{bcvekimpliescvek}, 
\[
        \frac{\p (w^+)^*}{\p z^*} = \alpha w^+
\]
and 
\[
        \frac{\p w^-}{\p z^*} = \alpha (w^-)^*.
\]
By Proposition 3.2.3.1 of \cite{conjbel}, this implies that the function $\tilde{g}$ defined by 
\[
        \tilde{g} := \frac{1}{\sqrt{1-\mu^2}}((w^+)^* + \mu w^+)
\]
solves
\[
        \frac{\p \tilde{g}}{\p z^*} = \mu \frac{\p (\tilde{g})^*}{\p z^*}
\]
and the function $\tilde{h}$ defined by 
\[
        \tilde{h} := \frac{1}{\sqrt{1-\mu^2}}( w^- + \mu (w^-)^*)
\]
solves
\[
        \frac{\p \tilde{h}}{\p z^*} = \mu \frac{\p (\tilde{h})^*}{\p z^*}.
\]
By direct computation, the function $f = p^+ (\tilde{g})^* + p^- \tilde{h}$ solves 
\[
    \dbar f = \mu \dbar \overline{f}.
\]

\end{proof}

\begin{comm}
    If the $w$ in the last theorem is known, then the function $f$ can be found to be 
    \[
        f = \frac{1}{\sqrt{1-\mu^2}} (w + \mu\overline{w}),
    \]
    as in the complex case.
\end{comm}

\begin{theorem}
Let $p$ be a positive real number. For real-valued $\mu\in W^{1,\infty}(\disk)$ such that there exists a constant $c$ satisfying $||\mu||_{L^\infty(\disk)}\leq c < 1$, define $\alpha \in L^\infty(\disk)$ by
\[
    \alpha := - \frac{1}{1- \mu^2} \dbar \mu.
\]
The function $f \in H^p_{conj,\mu}(\disk,\bc)$ if and only if $w := \frac{1}{\sqrt{1-\mu^2}} (f - \mu \overline{f}) \in H^p_{0, \alpha}(\disk,\bc)$. 
\end{theorem}

\begin{proof}
By Theorem \ref{thm: bcconjbelimpliesbcvek}, $f$ solves
\[
        \dbar f = \mu \dbar \overline{f}
\]
if and only if $w = \frac{1}{\sqrt{1-\mu^2}} (f - \mu \overline{f})$ solves 
\[
        \dbar w = \alpha \overline{w}.
\]

Suppose $f \in H^p_{conj, \mu}(\disk,\bc)$. Observe that 
\begin{align*}
    &\sup_{0 < r < 1} \int_0^{2\pi} ||w(re^{i\theta})||_{\bc}^p \,d\theta \\
    &= \sup_{0 < r < 1} \int_0^{2\pi} ||(1-\mu^2)^{-1/2} (f(re^{i\theta}) - \mu \overline{f(re^{i\theta})})||_{\bc}^p \,d\theta \\
    &\leq C_p ||(1-\mu^2)^{-1/2}||^p_{L^\infty(\disk)} \left( \sup_{0 < r < 1} \int_0^{2\pi} ||f(re^{i\theta})||_{\bc}^p \,d\theta + ||\mu||^p_{L^\infty(\disk)} \sup_{0 < r < 1} \int_0^{2\pi} ||f(re^{i\theta})||_{\bc}^p \,d\theta\right)  \\
    &< 2 C_p ||(1-\mu^2)^{-1/2}||^p_{L^\infty(\disk)}  \sup_{0 < r < 1} \int_0^{2\pi} ||f(re^{i\theta})||_{\bc}^p \,d\theta< \infty,
\end{align*}
where $C_p$ is a constant depending on $p$. Thus, $w \in H^p_{0,\alpha}(\disk,\bc)$. 

Now, suppose that $w \in H^p_{0, \alpha}(\disk,\bc)$. Then, $f = \frac{1}{\sqrt{1-\mu^2}} (w + \mu\overline{w})$, and 
\begin{align*}
&\sup_{0 < r < 1} \int_0^{2\pi} ||f(re^{i\theta})||_{\bc}^p \,d\theta \\
    &= \sup_{0 < r < 1} \int_0^{2\pi} ||(1-\mu^2)^{-1/2} (w(re^{i\theta}) + \mu \overline{w(re^{i\theta})})||_{\bc}^p \,d\theta \\
    &\leq C_p||(1-\mu^2)^{-1/2}||^p_{L^\infty(\disk)} \left( \sup_{0 < r < 1} \int_0^{2\pi} ||w(re^{i\theta})||_{\bc}^p \,d\theta + ||\mu||^p_{L^\infty(\disk)} \sup_{0 < r < 1} \int_0^{2\pi} ||w(re^{i\theta})||_{\bc}^p \,d\theta\right)  \\
    &\leq 2C_p ||(1-\mu^2)^{-1/2}||^p_{L^\infty(\disk)}  \sup_{0 < r < 1} \int_0^{2\pi} ||w(re^{i\theta})||_{\bc}^p \,d\theta< \infty,
\end{align*}
where $C_p$ is a constant depending on $p$. Therefore, $f \in H^p_{conj, \mu}(\disk,\bc)$. 
\end{proof}

\section{The General First-Order Elliptic Equation}\label{section: GFOE}

In this section, we define a bicomplex version of the general first-order elliptic equation using the bicomplex $\dbar$ operator from Definition \ref{bcdbardef}, define Hardy classes of bicomplex-valued solutions that extend their complex-valued analogues studied by Klimentov in \cite{Klim2016}, and show that, even in this most general form considered, the boundary behavior of the classic Hardy spaces is recovered. 

\subsection{Definition and Representation}

\begin{deff}
    Let $A, B \in L^q(\disk,\bc)$, $q>2$, and $\mu_1,\mu_2 : \disk \to \bc$ be such that there exists a real constant $c$ satisfying $||\mu_1||_{L^\infty(\disk,\bc)}+||\mu_2||_{L^\infty(\disk,\bc)} \leq c < 1$. We call any equation of the form 
    \[
        \dbar w = \mu_1  \p w + \mu_2 \dbar \overline{w} + Aw + B\overline{w}.
    \]
    a $\bc$-GFOE equation. 
\end{deff}

\begin{prop}
     For $A, B \in L^q(\disk,\bc)$, $q>2$, and $\mu_1,\mu_2 : \disk \to \bc$ such that there exists a real constant $c$ satisfying $||\mu_1||_{L^\infty(\disk,\bc)}+||\mu_2||_{L^\infty(\disk,\bc)} \leq c < 1$, every solution $w:\disk\to\C$ of
     \[
        \dbar w = \mu_1  \p w + \mu_2 \dbar \overline{w} + Aw + B\overline{w}.
    \]
    has the form $w = p^+ w^+ + p^- w^-$, where $w^+: \disk \to\mathbb{C}$ solves
    \[
        \frac{\p (w^+)^*}{\p z^*} = (\mu_1^+)^* \frac{\p (w^+)^*}{\p z} + (\mu_2^+)^*\frac{\p w^+}{\p z^*} + (A^+)^* (w^+)^* + (B^+)^* w^+
    \]
    and $w^-:\disk\to\mathbb{D}$ solves
    \[
        \frac{\p w^-}{\p z^*} = \mu_1^- \frac{\p w^-}{\p z} + \mu_2 \frac{\p (w^-)^*}{\p z^*} + A^- w^- + B^- (w^-)^*.
    \]
\end{prop}

\begin{proof}

By Proposition \ref{everybchasplusandminus} and Remark \ref{remark: idempotentdiffop}, any $w:\disk\to\bc$ that solves 
\[
        \dbar w = \mu_1  \p w + \mu_2 \dbar \overline{w} + Aw + B\overline{w}
    \]
 must satisfy
\begin{align*}
    &\left(p^+ \frac{\p}{\p z} + p^- \frac{\p}{\p z^*}\right)(p^+ w^+ + p^- w^-) \\
    &= (p^+\mu_1^+ + p^- \mu_1^-)  \left(p^+ \frac{\p}{\p z^*} + p^- \frac{\p}{\p z}\right)(p^+ w^+ + p^- w^-) \\
    &\quad\quad + (p^+\mu_2^+ + p^- \mu_2^-)  \left(p^+ \frac{\p}{\p z} + p^- \frac{\p}{\p z^*}\right)(p^+ (w^+)^* + p^- (w^-)^*) \\
    &\quad\quad + (p^+A^+ + p^- A^-)(p^+w^+ + p^- w^-) + (p^+B^+ + p^- B^-)(p^+(w^+)^* + p^- (w^-)^*).
\end{align*}
After simplifications, this is the same as
\begin{align*}
    &p^+\frac{\p w^+}{\p z} + p^- \frac{\p w^-}{\p z^*} \\
    &= p^+ \mu_1^+\frac{\p w^+}{\p z^*} + p^-\mu_1^- \frac{\p w^-}{\p z} + p^+\mu_2^+\frac{\p (w^+)^*}{\p z} + p^-\mu_2^- \frac{\p (w^-)^*}{\p z^*} \\
    & \quad \quad + p^+ A^+ w^+ + p^- A^- w^- + p^+ B^+ (w^+)^* + p^- B^- (w^-)^*.
\end{align*}
So, by associating the unique parts of the idempotent representation, we have
\begin{align*}
    \frac{\p w^+}{\p z} = \mu_1^+\frac{\p w^+}{\p z^*} + \mu_2^+ \frac{\p (w^+)^*}{\p z} + A^+ w^+ + B^+ (w^+)^*,
\end{align*}
and
\begin{align*}
    \frac{\p w^-}{\p z^*} = \mu_1^- \frac{\p w^-}{\p z} + \mu_2^- \frac{\p (w^-)^*}{\p z^*}  + A^- w^- + B^- (w^-)^*.
\end{align*}
Taking complex conjugates of both sides of the last equation gives us
\begin{align*}
     \frac{\p (w^+)^*}{\p z^*} = (\mu_1^+)^*\frac{\p (w^+)^*}{\p z} + (\mu_2^+)^* \frac{\p w^+}{\p z^*} + (A^+)^* (w^+)^* +  (B^+)^* w^+
\end{align*}

\end{proof}

    \subsection{Hardy Spaces}

\begin{deff}
For $A, B \in L^q(\disk,\bc)$, $q>2$, $\mu_1,\mu_2 : \disk \to \bc$ such that there exists a real constant $c$ satisfying $||\mu_1||_{L^\infty(\disk,\bc)}+||\mu_2||_{L^\infty(\disk,\bc)} \leq c < 1$, and $0 < p < \infty$, we define the $\bc$-GFOE-Hardy spaces $H^p_{\mu_1, \mu_2, A, B}(\disk,\bc)$ to be those functions $w: \disk\to\bc$ such that 
    \[
        \dbar w = \mu_1  \p w + \mu_2 \dbar \overline{w} + Aw + B\overline{w}.
    \]
and 
    \[
        ||w||_{H^p_{\bc}} := \left( \sup_{0< r < 1} \frac{1}{2\pi}\int_0^{2\pi} ||w(re^{i\theta})||_{\bc}^p\,d\theta\right)^{1/p} < \infty.
    \]
\end{deff}

\begin{comm}
    Note that if $\mu_1 \equiv 0 \equiv \mu_2$, then this class of functions is precisely the $\bc$-Vekua-Hardy classes from Definition \ref{deff: bcvekhardy} studied in \cite{BCHoiv, BCHarmVek}.
\end{comm}

\begin{theorem}\label{thm: bcgfoehardyrep}
    For $A, B \in L^q(\disk,\bc)$, $q>2$, $\mu_1,\mu_2 : \disk \to \bc$ such that there exists a real constant $c$ satisfying $||\mu_1||_{L^\infty(\disk,\bc)}+||\mu_2||_{L^\infty(\disk,\bc)} \leq c < 1$ and $0 < p < \infty$, $w \in H^p_{\mu_1,\mu_2,A,B}(D,\bc)$ if and only if $(w^+)^* \in H^p_{(\mu_1^+)^*, (\mu_2^+)^*, (A^+)^*, (B^+)^*}(\disk)$ and $w^- \in H^p_{\mu_1^-, \mu_2^-, A^-, B^-}(\disk)$. 
\end{theorem}

\begin{proof}
The proof of this theorem follows directly by the same argument used in the proofs of Theorems \ref{thm: bcbelhardyrep} and \ref{thm: bcconjbelhardyrep}. 
\end{proof}

\begin{theorem}
For $A, B \in L^q(\disk,\bc)$, $q>2$, $\mu_1,\mu_2 : \disk \to \bc$ such that there exists a real constant $c$ satisfying $||\mu_1||_{L^\infty(\disk,\bc)}+||\mu_2||_{L^\infty(\disk,\bc)} \leq c < 1$, and $0 < p < \infty$, every $w \in H^p_{\mu_1,\mu_2,A,B}(D,\bc)$ is an element of $L^m(\disk,\bc)$, for $0 < m < 2p$. 
\end{theorem}

\begin{proof}
The proof of this result is the same as Theorem \ref{thm: bcbelhpinbetterlp} except appeals are made to Theorem \ref{thm: gfoehardyinbetterlebesgue} instead of Theorem \ref{thm: belhpinbetterlp} and Theorem \ref{thm: bcgfoehardyrep} instead of Theorem \ref{thm: bcbelhardyrep}.

\end{proof}

    \subsection{Boundary Behavior}
    
      \begin{theorem}
       For $A, B \in L^q(\disk,\bc)$, $q>2$, $\mu_1,\mu_2 \in W^{1,q}(\disk,\bc)$ such that there exists a real constant $c$ satisfying $||\mu_1||_{L^\infty(\disk,\bc)}+||\mu_2||_{L^\infty(\disk,\bc)} \leq c < 1$, and $0 < p < \infty$, the following three statements follow:
	\begin{enumerate}
	\item Every $w \in H^p_{\mu_1,\mu_2,A,B}(\disk,\bc)$ has a nontangential boundary value $w_{nt} \in L^p(\p \disk, \bc)$ and
        \[
        \lim_{r \nearrow 1} \int_0^{2\pi} ||w_{nt}(e^{i\theta}) - w(re^{i\theta}) ||^p_{\bc} \, d\theta  = 0. 
        \]
	\item Every $w \in H^p_{\mu_1,\mu_2,A,B}(\disk,\bc)$ with nontangential boundary value $w_{nt} \in L^s(\p\disk, \bc)$, $s>p$, is an element of $H^s_{\mu_1,\mu_2,A,B}(\disk,\bc)$. 
	\item Every $w \in H^p_{\mu_1,\mu_2,A,B}(\disk,\bc)$ with nontangential boundary value $w_{nt}$ that vanishes on a set $E \subset \p \disk$ of positive measure is identically equal to zero. 
	\end{enumerate}
    \end{theorem}
    
        \begin{proof}
    This proof follows the same argument as Theorem \ref{thm: bcbetterbvbetterleb} and Theorem \ref{thm: bcconjbetterbvbetterleb} using Theorem \ref{thm: bcgfoehardyrep} and Theorem \ref{thm: gfoebvinbetterlebimpliesbetterhp}.
    \end{proof}

\section{Boundary Value Problems}\label{section: bvp}

In this final section, we consider the classic boundary value problems of complex analysis with the $\bc$-Beltrami equation. 

\subsection{The Bicomplex Schwarz Boundary Value Problem}

First, we consider a bicomplex version of the Schwarz boundary value problem. Our result is a direct generalization of the work of G. Harutyunyan \cite{BeltramiSchwarz} for the complex Beltrami equation and uses their result explicitly to construct the solution formulas. Note, in contrast to Section \ref{section: beleqn}, we only consider here the case of a constant coefficient $\mu$. However, we are able to consider a nonhomogeneous version of the bicomplex Beltrami equation, i.e., $\dbar w = \mu \p w + f$, where $f \not\equiv 0$. In Section \ref{section: beleqn}, we only considered the $f \equiv 0$ case of this equation. For other considerations of a bicomplex Schwarz boundary value problem, see \cite{BCSchwarz}.

\begin{theorem}
    The $\bc$-Schwarz problem 
        \[
            \begin{cases}
                \dbar w = \mu \p w + f, & \text{ in } \disk,\\
                \re\{w^+\}|_{\p\disk} = \gamma_1, \\
                \re\{w^-\}|_{\p\disk} = \gamma_2,\\
                \im\{w^+(0)\} = a_1, \\
                \im\{w^-(0)\} = a_2,
            \end{cases}
        \]
        for $\mu \in \bc$ such that $||\mu||_{\bc} \leq c < 1$, for some constant $c$, $f \in L^p(\disk, \bc)$, $p >2$, $\gamma_1,\gamma_2 \in C(\p \disk,\R)$, and $a_1,a_2 \in \R$, is uniquely solvable and the solution is 
        \[
             w = p^+ w^+ + p^- w^-,
        \]
        where $w^+$ is the complex conjugate of
        \begin{align*}
            (w^+(z))^* 
            &= \varphi_1(z) 
            + \sum_{k=0}^\infty (-1)^{k+1} \frac{1}{2\pi} \iint_{|\zeta|<1} \left( ((-\mu^+)^*)^k T^k((f^+)^*+(\mu^+)^*\varphi_1')(\zeta)\frac{\zeta + z}{\zeta(\zeta -z)} \right.\\
            &\quad\quad\quad\quad +  \left. ((-\mu^+))^k (T^k((f^+)^*+(\mu^+)^*\varphi_1')(\zeta))^*\frac{1-z \zeta^*}{\zeta^*(1-z\zeta^*)} \right)\,d\xi\,d\zeta,
        \end{align*}
        with $\zeta = \xi + i \eta$,
        \[
            \varphi_1(z) = \frac{1}{2\pi i} \int_{|\zeta|=1} \gamma_1(\zeta) \frac{\zeta + z}{\zeta -z} \frac{d\zeta}{\zeta} + i(-a_1),
        \]
        and $T(\cdot)$ is the operator defined by
        \[
            T(g)(z) := -\frac{1}{\pi} \iint_{|\zeta|<1} \frac{g(\zeta)}{(\zeta - z)^2} + \frac{(g(\zeta))^*}{(1-z\zeta^*)^2}\,d\xi\,d\eta,
        \]
        which solves the Schwarz problem
        \[
            \begin{cases}
                \frac{\p (w^+)^*}{\p z^*} = (\mu^+)^* \frac{\p (w^+)^*}{\p z} + (f^+)^*, & \text{ in }\disk,\\
                \re\{(w^+)^*\}|_{\p\disk} = \gamma_1,\\
                \im\{((w^+)(0))^*\} = -a_1,
            \end{cases}
        \]
        and $w^-$ is 
        \begin{align*}
            w^-(z) 
            &= \varphi_2(z) 
            + \sum_{k=0}^\infty (-1)^{k+1} \frac{1}{2\pi} \iint_{|\zeta|<1} \left( (-\mu^-)^k T^k(f^-+\mu^-\varphi_2')(\zeta)\frac{\zeta + z}{\zeta(\zeta -z)} \right.\\
            &\quad\quad\quad\quad +  \left. ((-\mu^-)^*)^k (T^k(f^-+\mu^-\varphi_2')(\zeta))^*\frac{1-z \zeta^*}{\zeta^*(1-z\zeta^*)} \right)\,d\xi\,d\zeta,
        \end{align*}
        with 
        \[
            \varphi_2(z) = \frac{1}{2\pi i} \int_{|\zeta|=1} \gamma_2(\zeta) \frac{\zeta + z}{\zeta -z} \frac{d\zeta}{\zeta} + ia_2,
        \]
        which solves the Schwarz problem
        \[
            \begin{cases}
                \frac{\p w^-}{\p z^*} = \mu^- \frac{\p w^-}{\p z} + f^-, & \text{ in } \disk,\\
                \re\{w^-\}|_{\p \disk} = \gamma_2,\\
                \im\{w^-(0)\} = a_2
            \end{cases}.
        \]
\end{theorem}

\begin{proof}
By direct computation, a function $w = p^+w^+ + p^- w^-:\disk\to\bc$ solves 
        \[
            \begin{cases}
                \dbar w = \mu \p w + f, & \text{ in } \disk,\\
                \re\{w^+\}|_{\p\disk} = \gamma_1, \\
                \re\{w^-\}|_{\p\disk} = \gamma_2,\\
                \im\{w^+(0)\} = a_1, \\
                \im\{w^-(0)\} = a_2,
            \end{cases}
        \]
        if and only if $w^+$ and $w^-$ solve
        \[
            \begin{cases}
                \frac{\p (w^+)^*}{\p z^*} = (\mu^+)^* \frac{\p (w^+)^*}{\p z} + (f^+)^*, & \text{ in }\disk,\\
                \re\{(w^+)^*\}|_{\p\disk} = \gamma_1,\\
                \im\{((w^+)(0))^*\} = -a_1,
            \end{cases}
        \] 
        and
        \[
            \begin{cases}
                \frac{\p w^-}{\p z^*} = \mu^- \frac{\p w^-}{\p z} + f^-, & \text{ in } \disk,\\
                \re\{w^-\}|_{\p \disk} = \gamma_2,\\
                \im\{w^-(0)\} = a_2
            \end{cases}.
        \]
        By Theorem \ref{thm: beltramischwarz}, $w^+$ and $w^-$ defined to be
        \begin{align}
            &(w^+(z))^* \label{pluseqn3}\\
            &= \varphi_1(z) 
            + \sum_{k=0}^\infty (-1)^{k+1} \frac{1}{2\pi} \iint_{|\zeta|<1} \left( ((-\mu^+)^*)^k T^k((f^+)^*+(\mu^+)^*\varphi_1')(\zeta)\frac{\zeta + z}{\zeta(\zeta -z)} \right.\label{pluseqn1}\\
            &\quad\quad\quad\quad +  \left. ((-\mu^+))^k (T^k((f^+)^*+(\mu^+)^*\varphi_1')(\zeta))^*\frac{1-z \zeta^*}{\zeta^*(1-z\zeta^*)} \right)\,d\xi\,d\zeta\label{pluseqn2}
        \end{align}
        and
        \begin{align}
            w^-(z) 
            &= \varphi_2(z) 
            + \sum_{k=0}^\infty (-1)^{k+1} \frac{1}{2\pi} \iint_{|\zeta|<1} \left( (-\mu^-)^k T^k(f^-+\mu^-\varphi_2')(\zeta)\frac{\zeta + z}{\zeta(\zeta -z)} \right.\label{minuseqn1}\\
            &\quad\quad\quad\quad +  \left. ((-\mu^-)^*)^k (T^k(f^-+\mu^-\varphi_2')(\zeta))^*\frac{1-z \zeta^*}{\zeta^*(1-z\zeta^*)} \right)\,d\xi\,d\zeta\label{minuseqn2}
        \end{align}
        are the unique solutions to 
         \[
            \begin{cases}
                \frac{\p (w^+)^*}{\p z^*} = (\mu^+)^* \frac{\p (w^+)^*}{\p z} + (f^+)^*, & \text{ in }\disk,\\
                \re\{(w^+)^*\}|_{\p\disk} = \gamma_1,\\
                \im\{((w^+)(0))^*\} = -a_1,
            \end{cases}
        \] 
        and
        \[
            \begin{cases}
                \frac{\p w^-}{\p z^*} = \mu^- \frac{\p w^-}{\p z} + f^-, & \text{ in } \disk,\\
                \re\{w^-\}|_{\p \disk} = \gamma_2,\\
                \im\{w^-(0)\} = a_2
            \end{cases}.
        \]
        Therefore, $w = p^+ w^+ + p^-w^-$, with $w^+$ as in \eqref{pluseqn3}, \eqref{pluseqn1}, \eqref{pluseqn2} and $w^-$ as in \eqref{minuseqn1}, \eqref{minuseqn2} is the unique solution of the Schwarz boundary value problem
         \[
            \begin{cases}
                \dbar w = \mu \p w + f, & \text{ in } \disk\\
                \re\{w^+\}|_{\p\disk} = \gamma_1, \\
                \re\{w^-\}|_{\p\disk} = \gamma_2,\\
                \im\{w^+(0)\} = a_1, \\
                \im\{w^-(0)\} = a_2.
            \end{cases}
        \]  
\end{proof}

\subsection{The Bicomplex Dirichlet Boundary Value Problem}

Next, we consider a bicomplex version of the Dirichlet boundary value problem. This problem was also considered for the complex Beltrami equation in \cite{BeltramiSchwarz}, and we make use of those results in constructing the solution formulas below. 

\begin{theorem}
    The $\bc$-Dirichlet problem 
    \[
        \begin{cases}
            \dbar w = \mu \p w + f, & \text{ in } \disk,\\
            w_b = \gamma,
        \end{cases}
    \]
    where $\mu \in \bc$ such that $||\mu||_\bc\leq c < 1$, for some constant $c$, $f \in L^p(\disk,\bc)$, $p>2$, and $\gamma \in C(\p\disk,\bc)$, is solvable if and only if
    \begin{align*}
        &\frac{1}{2\pi i} \int_{|\zeta| = 1} (\gamma^+(\zeta))^* \frac{2 + (-\mu^+)^* z^* \zeta^*}{1 + (-\mu^+)^* z^*\zeta^*} \frac{z^* \,d\zeta}{1-z^*\zeta} \\
        &= \sum_{k=0}^{\infty} (-1)^k ((-\mu^+)^*)^k \frac{1}{\pi} \iint_{|\zeta|<1} (f^+(\zeta))^* ((\zeta - z)^*)^k \frac{(z^*)^{k+1} \,d\xi\,d\eta}{(1-z^*\zeta)^{k+1}}
    \end{align*}
    and 
    \begin{align*}
        &\frac{1}{2\pi i} \int_{|\zeta| = 1} \gamma^-(\zeta) \frac{2 + (-\mu^-)z^* \zeta^*}{1 + (-\mu^-)z^*\zeta^*} \frac{z^* \,d\zeta}{1-z^*\zeta} \\
        &= \sum_{k=0}^{\infty} (-1)^k (-\mu^-)^k \frac{1}{\pi} \iint_{|\zeta|<1} f^-(\zeta) ((\zeta - z)^*)^k \frac{(z^*)^{k+1} \,d\xi\,d\eta}{(1-z^*\zeta)^{k+1}},
    \end{align*}
    and the solution is 
    \[
        w = p^+ w^+ + p^- w^-,
    \]
    where 
    \begin{align*}
        (w^+(z))^* &= \frac{1}{2\pi i} \int_{|\zeta| =1} \frac{(\gamma^+(\zeta))^*}{\zeta - z} \,d\zeta + \frac{1}{2\pi i} \int_{|\zeta| =1} \frac{(\gamma^+(\zeta))^*}{\zeta - z + c(\zeta - z)^*} \,d\zeta \\
        &\quad\quad- \frac{1}{\pi} \iint_{|\zeta|<1} \frac{(f^+(\zeta))^*}{\zeta - z + c(\zeta - z)^*}\,d\xi\,d\eta
    \end{align*}
    solves the Dirichlet problem 
    \[
        \begin{cases}
            \frac{\p (w^+)^*}{\p z^*} = (\mu^+)^* \frac{\p (w^+)^*}{\p z} + (f^+)^*, & \text{ in } \disk,\\
            (w^+)^*|_{\p\disk} = (\gamma^+)^*,
        \end{cases}
    \]
    and 
     \begin{align*}
        w^-(z) &= \frac{1}{2\pi i} \int_{|\zeta| =1} \frac{\gamma^-(\zeta)}{\zeta - z} \,d\zeta + \frac{1}{2\pi i} \int_{|\zeta| =1} \frac{\gamma^-(\zeta)}{\zeta - z + c(\zeta - z)^*} \,d\zeta \\
        &\quad\quad - \frac{1}{\pi} \iint_{|\zeta|<1} \frac{f^-(\zeta)}{\zeta - z + c(\zeta - z)^*}\,d\xi\,d\eta
    \end{align*}
    solves the Dirichlet problem 
    \[
        \begin{cases}
            \frac{\p w^-}{\p z^*} = \mu^- \frac{\p w^-}{\p z} + f^- ,& \text{ in } \disk,\\
            w^-|_{\p\disk} = \gamma^-
        \end{cases}.
    \]
\end{theorem}

\begin{proof}
    By direct computation, a function $w= p^+ w^+ + p^- w^- : \disk\to\bc$ solves 
    \[
        \begin{cases}
            \dbar w = \mu \p w + f, & \text{ in } \disk,\\
            w_b = \gamma
        \end{cases}
    \]
    if and only if $w^+$ and $w^-$ solve the Dirichlet problems 
    \[
        \begin{cases}
            \frac{\p (w^+)^*}{\p z^*} = (\mu^+)^* \frac{\p (w^+)^*}{\p z} + (f^+)^*, & \text{ in } \disk,\\
            (w^+)^*|_{\p\disk} = (\gamma^+)^*,
        \end{cases}
    \]
    and
    \[
        \begin{cases}
            \frac{\p w^-}{\p z^*} = \mu^- \frac{\p w^-}{\p z} + f^- ,& \text{ in } \disk,\\
            w^-|_{\p\disk} = \gamma^-
        \end{cases}.
    \]
    By Theorem \ref{thm: beltramidirichlet}, $w^+$ and $w^-$ solve 
    \[
        \begin{cases}
            \frac{\p (w^+)^*}{\p z^*} = (\mu^+)^* \frac{\p (w^+)^*}{\p z} + (f^+)^*, & \text{ in } \disk,\\
            (w^+)^*|_{\p\disk} = (\gamma^+)^*
        \end{cases}
    \]
    and
    \[
        \begin{cases}
            \frac{\p w^-}{\p z^*} = \mu^- \frac{\p w^-}{\p z} + f^- ,& \text{ in } \disk,\\
            w^-|_{\p\disk} = \gamma^-
        \end{cases}
    \]
    if and only if 
    \begin{align*}
        &\frac{1}{2\pi i} \int_{|\zeta| = 1} (\gamma^+(\zeta))^* \frac{2 + (-\mu^+)^* z^* \zeta^*}{1 + (-\mu^+)^* z^*\zeta^*} \frac{z^* \,d\zeta}{1-z^*\zeta} \\
        &= \sum_{k=0}^{\infty} (-1)^k ((-\mu^+)^*)^k \frac{1}{\pi} \iint_{|\zeta|<1} (f^+(\zeta))^* ((\zeta - z)^*)^k \frac{(z^*)^{k+1} \,d\xi\,d\eta}{(1-z^*\zeta)^{k+1}}
    \end{align*}
    and 
    \begin{align*}
        &\frac{1}{2\pi i} \int_{|\zeta| = 1} \gamma^-(\zeta) \frac{2 + (-\mu^-)z^* \zeta^*}{1 + (-\mu^-)z^*\zeta^*} \frac{z^* \,d\zeta}{1-z^*\zeta} \\
        &= \sum_{k=0}^{\infty} (-1)^k (-\mu^-)^k \frac{1}{\pi} \iint_{|\zeta|<1} f^-(\zeta) ((\zeta - z)^*)^k \frac{(z^*)^{k+1} \,d\xi\,d\eta}{(1-z^*\zeta)^{k+1}}.
    \end{align*}
    
\end{proof}

\section*{Acknowledgements}
The author thanks the anonymous referees for their careful reading of the manuscript, thoughtful comments, and helpful suggestions which have improved the quality of this work.

\printbibliography

\end{document}